\documentclass[11pt,a4paper,leqno]{amsart}

\title
[Filtering of second order generalized stochastic processes]
{Filtering of second order generalized stochastic processes corrupted by additive noise}

\author[P. Wahlberg]{Patrik Wahlberg}

\address{Dipartimento di Scienze Matematiche, Politecnico di Torino, Corso Duca degli Abruzzi 24,
10129 Torino, Italy}

\email{patrik.wahlberg[AT]polito.it}

\usepackage{latexsym}
\usepackage[a4paper]{geometry}
\usepackage{amsmath}
\usepackage{amssymb}
\usepackage{amsthm}
\usepackage{bbm}
\usepackage{amsfonts}
\usepackage{mathrsfs}
\usepackage{calc}
\usepackage{cite}
\usepackage{color}
\usepackage{graphicx}
\usepackage{enumerate}
\usepackage{bm}

\numberwithin{equation}{section}          

\newtheorem{thm}{Theorem}
\numberwithin{thm}{section}

\newcommand{\rubrik}{}
\newtheorem{prop}[thm]{Proposition}
\newtheorem{cor}[thm]{Corollary}
\newtheorem{lem}[thm]{Lemma}

\theoremstyle{definition}

\newtheorem{defn}[thm]{Definition}
\newtheorem{example}[thm]{Example}

\theoremstyle{remark}

\newtheorem{rem}[thm]{Remark}              

\newcommand{\scal}[2]{\langle #1,#2\rangle}
\newcommand{\pd}[1] {\partial ^#1}

\newcommand{\bP}{\mathbf P}
\newcommand{\bE}{\mathbf E}
\newcommand{\ro}{\mathbf R}
\newcommand{\no}{\mathbf N}
\newcommand{\rr}[1]{\mathbf R^{#1}}

\newcommand{\sro}[1]{\mathbf S}
\newcommand{\nn}[1]{\mathbf N^{#1}}

\newcommand{\zz}[1]{\mathbf Z^{#1}}
\newcommand{\zo}{\mathbf Z}

\newcommand{\co}{\mathbf C}

\newcommand{\dd}{\mathrm {d}}

\newcommand{\ep}{\varepsilon}
\newcommand{\fy}{\varphi}

\newcommand{\cdo}{\, \cdot \, }

\newcommand{\supp}{\operatorname{supp}}
\newcommand{\esssupp}{\operatorname{ess\, supp}}

\newcommand{\eabs}[1]{\langle #1\rangle}

\newcommand{\ran}{\operatorname{ran}}

\newcommand{\rB}{\operatorname{B}}

\newcommand{\mascP}{\mathscr P}

\newcommand{\cS}{\mathscr{S}}

\newcommand{\cH}{\mathscr{H}}
\newcommand{\cM}{\mathscr{M}}
\newcommand{\cT}{\mathscr{T}}

\newcommand{\cB}{\mathscr{B}}
\newcommand{\cD}{\mathscr{D}}
\newcommand{\cF}{\mathscr{F}}

\newcommand{\cK}{\mathscr{K}}

\newcommand{\cL}{\mathscr{L}}

\newcommand{\J}{\mathcal{J}}
\newcommand{\wt}{\widetilde}
\newcommand{\wh}{\widehat}
\newcommand{\re}{{\rm Re}\,}

\newcommand{\id}{{\rm id}\,}
\def\la{\langle}
\def\ra{\rangle}
\newcommand{\leqs}{\leqslant}
\newcommand{\geqs}{\geqslant}
\newcommand{\gsp}{\operatorname{GSP}}



\begin{document}

\maketitle

\begin{abstract}
We treat the optimal linear filtering problem for a sum of two second order uncorrelated 
generalized stochastic processes. 
This is an operator equation involving covariance operators. 
We study both the wide-sense stationary case and the non-stationary case. 
In the former case the equation 
simplifies into a convolution equation. 
The solution is the Radon--Nikodym derivative between  
non-negative tempered Radon measures, for signal and signal plus noise respectively,
in the frequency domain. 
In the non-stationary case we work with pseudodifferential operators with symbols in Sj\"ostrand modulation spaces
which admits the use of its spectral invariance properties. 
\end{abstract}

\par

\section{Introduction}\label{sec:intro}

The paper concerns zero mean second order generalized stochastic processes. 
These are defined as linear continuous maps from a space of test functions defined on $\rr d$ into a Hilbert space of finite variance random variables. 
Given a sum of two such uncorrelated generalized stochastic processes
we study the optimal filtering problem, whose goal is to find the
linear operator that recovers one of them with the minimum mean square error. 

If $u$ denotes the useful signal and $w$ the uncorrelated noise, the equation for the optimal linear operator $F$ reads
\begin{equation}\label{eq:filtereq}
\cK_u = F ( \cK_u + \cK_w)
\end{equation}
where $\cK_u$ and $\cK_w$ are the covariance operators of $u$ and $w$ respectively. 
These are non-negative continuous linear operators from $C_c^\infty(\rr d)$ to the distributions $\cD'(\rr d)$. 

If both generalized stochastic processes are wide-sense stationary then the covariance operators are translation invariant
which means that they are convolution operators. 
Their covariance kernels are then Fourier transforms of non-negative tempered Radon measures $\mu_u$ and $\mu_w$ respectively defined on $\rr d$. 
It is then natural to impose the operator $F$ in \eqref{eq:filtereq} to be a convolution operator, that is $F = f *$. 

Our first result is the determination of the optimal convolution operator in the wide-sense stationary case. 
In the frequency domain it turns out to be the Radon--Nikodym derivative 
\begin{equation*}
\wh f = \frac{\dd \mu_u}{\dd (\mu_u + \mu_w)}
\end{equation*}
which is a function in $L^\infty( \mu_u + \mu_w )$ that satisfies $0 \leqs \wh f \leqs 1$ almost everywhere. 
This result generalizes and formalizes widely established engineering intuition for the optimal filter for wide-sense stationary processes. 
The functional framework for equation  \eqref{eq:filtereq} in this case are Hilbert spaces $\cF L^2(\mu)$, that is tempered distributions with Fourier transforms
that belong to $L_{\rm{loc}}^2(\mu)$ and are square integrable with respect to a non-negative tempered Radon measure $\mu$. 
In fact for wise-sense stationary generalized stochastic processes the covariance operators act continuously on such spaces for certain $\mu$. 

Secondly we study equation \eqref{eq:filtereq} under the assumption that the generalized stochastic processes are non-stationary
which is a less well defined problem. We need frameworks and restrictions to be able to formulate solutions. 

As an intermediate step from wide-sense stationarity to non-stationarity we first impose the restriction that the covariance operators $\cK_u$ and $\cK_w$ commute, which holds in the former case, and are continuous and non-negative on a Hilbert space of distributions on $\rr d$.  
Then we may solve \eqref{eq:filtereq} using the spectral theorem and its associated functional calculus as
\begin{equation*}
F = \int_{\co} f(z) \dd \Pi (z) \in \cL( \cH )
\end{equation*}
where $\Pi$ is a projection-valued measure compactly supported in the first quadrant in $\co$, and where
\begin{equation*}
f(z) = \frac{f_u(z)}{f_u(z) + f_w(z)} \, \chi_{\supp f_u}
\end{equation*}
with 
\begin{equation*}
\cK_u = \int_{\co} f_u(z) \dd \Pi (z), \quad
\cK_w = \int_{\co} f_w(z) \dd \Pi (z). 
\end{equation*}

Finally we relax the commutativity of $\cK_u$ and $\cK_w$ and study equation \eqref{eq:filtereq}
as an equation for pseudodifferential covariance operators of non-stationary generalized stochastic processes. 
We work in the functional framework of modulation spaces for the Weyl symbols of the operators
and for the spaces on which operators act. 
Our goal is to extend the analysis for the wide-sense stationary case. 
We show embeddings of modulation spaces in the spaces $\cF L^2(\mu)$. 
We also show that the Weyl symbols $1 \otimes \mu$, where 
$\mu$ is a non-negative tempered Radon measure, of the covariance operators in the wide-sense stationary case 
do belong to modulation spaces $M_\omega^{\infty,1}(\rr {2d})$ for certain weights $\omega$ defined on $\rr {4d}$. 
This implies that the corresponding operators map between certain modulation spaces, which gives a functional framework for the equation \eqref{eq:filtereq}. 
The weight $\omega$ is however in general decreasing which is too weak for the exploitation of the powerful methods 
based on Gr\"ochenig's and Sj\"ostrand's results on the Wiener property of the symbol class $M_{1 \otimes \omega}^{\infty,1}(\rr {2d})$
where $\omega(X) = (1+|X|)^r$ for $X \in \rr {2d}$ and $r \geqs 0$. 

Imposing covariance operators to have Weyl symbols in $M_{1 \otimes \omega}^{\infty,1}(\rr {2d})$ admits 
white noise but excludes certain other wide-sense stationary generalized stochastic processes, for example derivatives of white noise. 
In this framework we show the following result
which is a rather immediate consequence of results in \cite{Wahlberg1}. 
If $\cK_u$ and $\cK_w$ have Weyl symbols in $M_{1 \otimes \omega}^{\infty,1}(\rr {2d})$ for some $r \geqs 0$
and $\cK_u + \cK_w$ is invertible as an operator on $L^2(\rr d)$, then the optimal filter $F$ solving \eqref{eq:filtereq}
is again a Weyl pseudodifferential operator with symbol in $M_{1 \otimes \omega}^{\infty,1}(\rr {2d})$. 
If $ r > 2 d$ then the Gabor coefficients of the Weyl symbols of $F$ and $\cK_u$ are related by a multiplication of 
an infinite matrix with polynomial off-diagonal decay of order smaller than $r/2 - d$. 

Finally we discuss a few operator theoretic observations concerning equation \eqref{eq:filtereq} considered as an 
equation for bounded linear operators on a Hilbert space on which $\cK_u$ and $\cK_w$ are non-negative, 
and $\cK_u + \cK_w$ is allowed to be non-invertible. 
Douglas' lemma then says that there is a bounded linear operator $F$ that solves \eqref{eq:filtereq}, with certain uniqueness properties, 
provided $\ran \cK_u \subseteq \ran( \cK_u + \cK_w )$. 

The analysis in this work is second order, which explains the assumption that the processes are uncorrelated rather than statistically independent. 

The optimal filtering problem for (generalized) stochastic processes has a long and eclectic history going back to Kolmogorov \cite{Kolmogorov1} and Wiener \cite{Wiener1}, cf. \cite{Fomin1,Fomin2,Hlawatsch1,Papoulis1,Vantrees1,Wahlberg1,Weinstein1}. 
It has been studied mostly under the assumption of wide-sense stationarity. 
The time domain has been either $\zo$ or $\ro$, and the filter has often been subject to the constraint to be \emph{causal}. 
This means that the convolutor $f$ has support on $\ro_+$ and it is natural in engineering applications. 
Usually the term Wiener filter refers to this constraint.  
In this paper we do not restrict the supports of filter kernels. 
Therefore we do not treat causal filters, neither in the wide-sense stationary nor in the non-stationary case.

The paper is organized as follows. 
Section \ref{sec:prelim} is quite long and contains notations, and background material
on Radon measures, Weyl pseudodifferential operators, modulation spaces, and Gabor frames. 
It also specifies the framework of generalized stochastic processes. 
In Section \ref{sec:optimalfilter} we deduce the equation \eqref{eq:filtereq}
for the optimal filter operator,
and Section \ref{sec:optimalwss} treats its solution for 
wide-sense stationary generalized stochastic processes. 

In Section \ref{sec:optfiltnonstat} we study equation \eqref{eq:filtereq} for non-stationary 
generalized stochastic processes. 
First we keep the feature of commutating covariance operators which holds in the wide-sense stationary case. 
Then a solution to \eqref{eq:filtereq} may be determined using the spectral theorem. 
Finally we relax the commutativity and study the equation \eqref{eq:filtereq} as an equation for 
pseudodifferential operators with symbols in modulation spaces.

\section{Preliminaries}\label{sec:prelim}

\subsection{Notations}\label{subsec:notations}

The symbol $\rB_r$ denotes the ball in $\rr d$ with center at the origin and radius $r > 0$.
The notation $K \Subset \rr d$ means that $K$ is compact. 
We use $\ro_+$ for the non-negative real numbers, 
and $\chi_A$ is the indicator function of a subset $A \subseteq \rr d$. 
We write $f (x) \lesssim g (x)$ provided there exists $C>0$ such that $f (x) \leqs C \, g(x)$ for all $x$ in the domain of $f$ and of $g$. 
If $f (x) \lesssim g (x) \lesssim f(x)$ then we write $f \asymp g$. 
The partial derivative $D_j = - i \partial_j$, $1 \leqs j \leqs d$, acts on functions and distributions on $\rr d$, 
with extension to multi-indices as $D^\alpha = i^{-|\alpha|} \partial^\alpha$ for $\alpha \in \nn d$. 
We use the bracket $\eabs{x} = (1 + |x|^2)^{\frac12}$ for $x \in \rr d$. 
The space of bounded linear operators on a Banach space $X$ is denoted $\cL(X)$, 
and inclusions $X \subseteq Y$ of Banach spaces understand embeddings, that is continuity. 
Peetre's inequality with optimal constant (cf. \cite[Lemma~2.1]{Rodino1}) is 
\begin{equation}\label{eq:Peetre}
\eabs{x+y}^s \leqs \left( \frac{2}{\sqrt{3}} \right)^{|s|} \eabs{x}^s\eabs{y}^{|s|}\qquad x,y \in \rr d, \quad s \in \ro. 
\end{equation}

If $1 \leqs p \leqs \infty$ then the conjugate exponent $p' \in [1, \infty]$ satisfies $\frac1p + \frac{1}{p'} = 1$. 
The normalization of the Fourier transform is
\begin{equation*}
 \cF f (\xi )= \widehat f(\xi ) = (2\pi )^{-\frac d2} \int _{\rr
{d}} f(x)e^{-i\scal  x\xi }\, \dd x, \qquad \xi \in \rr d, 
\end{equation*}
for $f\in \cS(\rr d)$ (the Schwartz space), where $\scal \cdo \cdo$ denotes the scalar product on $\rr d$. 
We have for $f, g \in \cS(\rr d)$
\begin{equation}\label{eq:convolutionFourier}
\wh{f * g} = (2\pi )^{\frac d2} \wh f \, \wh g
\end{equation}
and this identity extends to $f \in \cS'(\rr d)$ (the tempered distributions) and $g \in \cS(\rr d)$,
with the Fourier transform defined on $\cS'(\rr d)$ 
as $( \wh f, \wh g) = (f,g)$ for $f \in \cS'(\rr d)$ and $g \in \cS(\rr d)$.

The conjugate linear (antilinear) action of a distribution $u$ on a test function $\phi$ is written $(u,\phi)$, consistent with the $L^2$ inner product $(\cdo ,\cdo ) = (\cdo ,\cdo )_{L^2}$ which is conjugate linear in the second argument. 
Translation of a function or a distribution $f$ is denoted $T_x f(y) = f(y-x)$ for $x,y \in \rr d$, 
and modulation as $M_\xi f(x) = e^{i \la x, \xi \ra} f(x)$ for $x,\xi \in \rr d$. 

\subsection{Radon measures}\label{subsec:measures}

We use non-negative Radon measures on $\rr d$ \cite{Folland1}. 
By the Riesz representation theorem \cite[Theorem~7.2]{Folland1}, \cite[Theorem~2.14]{Rudin2} we may regard such a measure either
as a regular non-negative Borel measure, that is a regular $\sigma$-additive function whose domain is the Borel $\sigma$-algebra on $\rr d$, 
denoted $\cB(\rr d)$,
or equivalently we may regard it as a non-negative linear functional on $C_c(\rr d)$ which denotes the space of compactly supported continuous functions, with certain regularity properties. 
The latter description implies that the measure $\mu$ satisfies an estimate of the form 
\begin{equation*}
|(\mu, \fy)| \leqs C_K \sup_{x \in \rr d} |\fy (x)|, \quad \fy \in C_c(K),
\end{equation*}
with $C_K > 0$ for each $K \Subset \rr d$, and the non-negativity means 
$(\mu, \fy) \geqs 0$ when $\fy \geqs 0$. 
We adopt the convention that $\mu$ is an antilinear functional, and then using the former description of a non-negative Radon measure $\mu$
we may write
\begin{equation*}
(\mu, \fy) = \int_{\rr d} \overline{\fy (x)} \, \dd \mu (x), \quad \fy \in C_c(K).
\end{equation*}

If the measure $\mu$ is finite then it extends uniquely to an antilinear functional on 
the space $C_0(\rr d)$ which denotes the space of continuous functions that vanish at infinity \cite[Theorem~7.17]{Folland1}. 
The space $C_0(\rr d)$ is the completion of $C_c(\rr d)$ with respect to the uniform topology of $L^\infty(\rr d)$. 

If 
\begin{equation}\label{eq:temperedmeasure}
\int_{\rr d} \eabs{x}^{-s}  \, \dd \mu (x) < \infty
\end{equation}
for some $s \geqs 0$ then the measure $\mu$ is said to be tempered. 
The space of non-negative tempered Radon measures is denoted $\cM_+ (\rr d) \subseteq \cS'(\rr d)$.
An example is Lebesgue measure which is tempered for any $s > d$. 
A measure $\mu \in \cM_+ (\rr d)$ extends uniquely to a continuous antilinear functional on 
$C_{s,0}(\rr d)$ which denotes the space of continuous functions $f$ that vanishes at infinity quicker than $\eabs{\cdot}^{-s}$:  
\begin{equation*}
\lim_{|x| \to \infty} | f(x)| \eabs{x}^s = 0
\end{equation*}
equipped with the weighted supremum norm $\| f \eabs{\cdot}^s \|_{L^\infty(\rr d)}$. 

If $\mu \in \cM_+ (\rr d)$ then we define the Hilbert space
$\cF L^2(\mu)$ as the subspace of $f \in \cS'(\rr d)$ such that $\wh f \in L_{{\rm loc}}^2(\mu)$
and 
\begin{equation*}
\| f \|_{\cF L^2(\mu)} =\left(  \int_{\rr d} | \wh f (\xi) |^2 \, \dd \mu (\xi) \right)^{\frac12} < \infty.
\end{equation*}
It follows from \cite[Proposition~7.1]{Folland1} that $C_c^\infty(\rr d) \subseteq L^2(\mu)$ is 
a dense subspace. 
From \eqref{eq:temperedmeasure} 
we get the estimate for $f \in \cS(\rr d)$
\begin{equation*}
\| f \|_{L^2(\mu)} = \left( \int_{\rr d} | f (\xi) |^2 \, \dd \mu (\xi) \right)^{\frac12}
\lesssim \sup_{\xi \in \rr d} \eabs{\xi}^{\frac{s}{2}} | f(\xi)|
\end{equation*}
for some $s \geqs 0$, which implies that $\cS(\rr d) \subseteq L^2(\mu)$ is a continuous inclusion. 
It follows that $\cS(\rr d) \subseteq \cF L^2(\mu)$ is a dense and continuous inclusion. 
Note however that the inclusion $\cS(\rr d) \subseteq \cF L^2(\mu)$ is not guaranteed to be injective, since 
$\wh f = 0$ in $\cF L^2(\mu)$ means that $\wh f = 0$ $\mu$-a.e. but $\wh f \in \cS(\rr d) \setminus \{ 0 \}$ may
hold, for instance if $\supp \mu \subseteq \rr d$ is compact.  

The inclusion $\cF L^2(\mu) \subseteq \cS'(\rr d)$ is likewise continuous and dense.  
In fact the density follows from the density of $\cS(\rr d) \subseteq \cS'(\rr d)$ \cite[Corollary~V.3.1]{Reed1}. 
The inclusion $\cF L^2(\mu) \subseteq \cS'(\rr d)$ is also injective, as opposed to the possible non-injectivity of $\cS(\rr d) \subseteq \cF L^2(\mu)$. 
Nevertheless if $\supp \mu \subseteq \rr d$ has non-empty interior
then the possibly non-injective inclusion $\cS(\rr d) \subseteq \cF L^2(\mu)$
can be modified such that it becomes injective, if the test function space is modified as 
\begin{equation*}
\cF C_c^\infty(\supp \mu) \subseteq \cF L^2(\mu).
\end{equation*}

We have $H_{\frac{s}{2}}^\infty (\mu) \subseteq \cF L^2(\mu)$
provided  \eqref{eq:temperedmeasure} holds true,
where 
\begin{equation*}
H_{t}^\infty (\mu) = \{ \fy \in \cS'(\rr d): \quad \wh \fy \eabs{\cdot}^t \in L^\infty(\mu) \}, 
\quad t \in \ro,
\end{equation*}
is a scale of Sobolev spaces with respect to $\mu \in \cM_+ (\rr d)$. 
If $\mu$ is Lebesgue measure we write $H_t^\infty(\mu) = H_t^\infty (\rr d)$. 
When $t = 0$ we have $H_0^\infty(\rr d) = \cF L^\infty (\rr d)$ which is the space of pseudo-measures.
It can be identified with the topological dual of the Fourier algebra $\cF L^1(\rr d)$ \cite{Katznelson1,Larsen1}. 

The same conclusion holds for $\cF L^\infty (\mu) = H_0^\infty (\mu)$ if $\mu \in \cM_+ (\rr d)$. 
In fact by \cite[Theorem~6.15]{Folland1} the dual $(L^1 (\mu))'$ can be identified isometrically 
with $L^\infty (\mu)$ via the duality
\begin{equation*}
L^1 (\mu) \times L^\infty (\mu) \ni (f, g) \mapsto \int_{\rr d} f(x) \overline{g(x)} \, \dd \mu (x).  
\end{equation*}
We can identify $L^\infty (\mu) \subseteq \cS'(\rr d)$ as a subspace as
\begin{equation*}
(f, \fy) = \int_{\rr d} f(x) \overline{ \fy(x) } \, \dd \mu (x), \quad f \in L^\infty (\mu), \quad \fy \in \cS(\rr d). 
\end{equation*}
The inclusion $L^\infty (\mu) \subseteq \cS'(\rr d)$ is continuous and injective. 
The Fourier transform of $f \in L^\infty (\mu)$, considered as a tempered distribution $f \in \cS'(\rr d)$, 
equals 
\begin{equation*}
(\wh f, \wh \fy) = \int_{\rr d} f(x) \overline{ \fy(x) } \, \dd \mu (x), \quad \fy \in \cS(\rr d). 
\end{equation*}
It follows that $\| f \|_{\cF L^\infty (\mu)} = \| \wh f \|_{L^\infty (\mu)}$
and $\cF L^\infty (\mu) = \left( \cF L^1 (\mu) \right)'$, cf. \cite[Theorem~4.2.2]{Larsen1}.
The space $H_{0}^\infty (\mu) = \cF L^\infty(\mu)$ for $\mu \in \cM_+ (\rr d)$ will play an essential role in this paper.

\subsection{Weyl pseudodifferential operators}\label{subsec:weyloperators}

If $a \in \cS(\rr {2d})$ is a Weyl symbol then the Weyl pseudodifferential operator \cite{Folland1,Hormander1,Shubin1} is defined as 
\begin{equation}\label{eq:weylquantization}
a^w(x,D) f(x)
= (2\pi)^{-d}  \int_{\rr {2d}} e^{i \langle x-y, \xi \rangle} a \left(\frac{x+y}{2},\xi \right) \, f(y) \, \dd y \, \dd \xi, \quad f \in \cS(\rr d), 
\end{equation}
and $a^w(x,D): \cS(\rr d) \to \cS(\rr d)$ is continuous. 
By the invariance of $\cS'(\rr {2d})$ under linear invertible coordinate transformations and partial Fourier transforms, 
the Weyl correspondence extends to $a \in \cS'(\rr {2d})$ in which case $a^w(x,D): \cS(\rr d) \to \cS' (\rr d)$ is continuous. 

If $a \in \cS'(\rr {2d})$ then 
\begin{equation}\label{eq:wignerweyl}
( a^w(x,D) f, g) = (2 \pi)^{-\frac{d}{2}} ( a, W(g,f) ), \quad f, g \in \cS(\rr d), 
\end{equation}
where the cross-Wigner distribution \cite{Folland1,Grochenig1} is defined as 
\begin{equation}\label{eq:WignerSchwartz}
\begin{aligned}
W(g,f) (x,\xi) 
& = (2 \pi)^{-\frac{d}{2}} \int_{\rr d} g (x+y/2) \overline{f(x-y/2)} e^{- i \la y, \xi \ra} \dd y \\
& = \cF_2 \left( ( g \otimes \overline f) \circ T \right)(x,\xi), \quad (x,\xi) \in T^* \rr d.
\end{aligned}
\end{equation}
Here $\cF_2$ denotes the partial Fourier transform with respect to the second $\rr d$ variable in $\rr {2d}$,
and $T$ is the matrix 
\begin{equation}\label{eq:T}
T = 
\left(
\begin{array}{ll}
I_d & \frac12 I_d \\
I_d & - \frac12 I_d
\end{array}
\right) \in \rr {2d \times 2d}. 
\end{equation}

Conversely, by the Schwartz kernel theorem, for any continuous linear operator $\cK: \cS(\rr d) \to \cS' (\rr d)$ there exists 
a kernel $k \in  \cS' (\rr {2d})$
and a Weyl symbol $a \in  \cS' (\rr {2d})$ such that 
\begin{equation*}
( \cK f,g) = (k, g \otimes \overline f) 
= (2 \pi)^{-\frac{d}{2}} ( a, W(g,f) ), \quad f,g \in \cS(\rr d). 
\end{equation*}
From this we may extract the relation between the Schwartz kernel and the Weyl symbol of a linear continuous operator 
$\cK: \cS(\rr d) \to \cS'(\rr d)$: 
\begin{equation}\label{eq:KernelWeylsymbol}
k = (2 \pi)^{-\frac{d}{2}} ( \cF_2^{-1} a ) \circ T^{-1}
\quad \Longleftrightarrow \quad 
a = (2 \pi)^{\frac{d}{2}} \cF_2 \left( k  \circ T \right).
\end{equation}
%

\subsection{The short-time Fourier transform, modulation spaces and Gabor frames}
\label{subsec:modspace}

Let $\fy \in \cS(\rr d) \setminus\{ 0 \}$. 
The short-time Fourier transform (STFT) of a tempered distribution $u \in \cS'(\rr d)$ is defined by 
\begin{equation}\label{eq:STFT}
V_\fy u (x,\xi) = (2\pi )^{-\frac d2} (u, M_\xi T_x \fy) = \cF (u T_x \overline \fy)(\xi), \quad x,\xi \in \rr d. 
\end{equation}
The function $V_\fy u$ is smooth and polynomially bounded \cite[Theorem~11.2.3]{Grochenig1} as
\begin{equation}\label{eq:STFTtempered}
|V_\fy u (x,\xi)| \lesssim \eabs{(x,\xi)}^{k}, \quad (x,\xi) \in T^* \rr d, 
\end{equation}
for some $k \geqs 0$. 
We have $u \in \cS(\rr d)$ if and only if
\begin{equation*}
|V_\fy u (x,\xi)| \lesssim \eabs{(x,\xi)}^{-k}, \quad (x,\xi) \in T^* \rr d, \quad \forall k \geqs 0.  
\end{equation*}

The inverse transform is given by
\begin{equation}\label{eq:STFTinverse}
u = (2\pi )^{-\frac d2} \iint_{\rr {2d}} V_\fy u (x,\xi) M_\xi T_x \fy \, \dd x \, \dd \xi
\end{equation}
provided $\| \fy \|_{L^2} = 1$, with action under the integral understood, that is 
\begin{equation}\label{eq:moyal}
(u, f) = (V_\fy u, V_\fy f)_{L^2(\rr {2d})}
\end{equation}
for $u \in \cS'(\rr d)$ and $f \in \cS(\rr d)$, cf. \cite[Theorem~11.2.5]{Grochenig1}. 

A \emph{weight} on $\rr d$ is a positive function $\omega \in  L^\infty _{{\rm loc}}(\rr d)$
such that $1/\omega \in  L^\infty _{{\rm loc}}(\rr d)$.
The weight $\omega$ is called $v$-moderate if there is a positive locally bounded function
$v$ such that
\begin{equation}\label{eq:weightmoderate}
\omega(x+y) \leqs C \omega(x)v(y),\quad x,y \in\rr{d},
\end{equation}
for some constant $C \geqs 1$. 
The set $\mascP (\rr d)$ consists of weights that are
$v$-moderate for a polynomially bounded weight, that is a weight of
the form $v(x) = \eabs{x}^s$ with $s \geqs 0$.

Modulation spaces were introduced by Feichtinger 1983 and have been studied thoroughly from many points of view.
Gröchenig's book \cite{Grochenig1} is an excellent source for their basic properties. 

\begin{defn}\label{def:modsp}
Let $\omega \in \mascP (\rr {2d})$, let $p,q \in [1, \infty]$ and let $\fy \in \cS(\rr d) \setminus \{ 0 \}$.  
The modulation space $M_\omega^{p,q} (\rr d)$ is the Banach subspace of $\cS'(\rr d)$
defined by the norm
\begin{equation}\label{eq:mospnorm}
\| u \|_{M_\omega^{p,q}} = \left( \int_{\rr d} \left( \int_{\rr d} |V_\fy u (x,\xi)|^p \, \omega (x,\xi)^p \, \dd x \right)^{\frac{q}{p}} \, \dd \xi \right)^{\frac1{q}}
= \| (V_\fy u) \omega \|_{L^{p,q}}
\end{equation}
when $p, q < \infty$ and the usual modifications otherwise. 
Here $L^{p,q}(\rr {2d})$ is a mix-normed Lebesgue space \cite{Grochenig1}.
\end{defn}

The modulation spaces are independent of $\fy \in \cS(\rr d) \setminus \{ 0 \}$ and increase with the indices as 
\begin{equation}\label{eq:modspnested}
M_\omega^{1,1} \subseteq M_\omega^{p,q} \subseteq M_\omega^{r,s} \subseteq M_\omega^{\infty,\infty}, \quad 
1 \leqs p \leqs r, \quad  
1 \leqs q \leqs s.
\end{equation}
We write $M_\omega^{p,p} = M_\omega^{p}$ and $M_\omega^{p,q} = M^{p,q}$ if $\omega = 1$. 
If $\omega(x,\xi) = \eabs{x}^t \eabs{\xi}^s$ for $x,\xi \in \rr d$ and $t,s \in \ro$ then we write 
$M_\omega^{p,q}(\rr d) = M_{t,s}^{p,q}(\rr d)$.
Note that $\omega \in \mascP (\rr {2d})$, where the polynomial moderateness is a consequence of 
\eqref{eq:Peetre}. 
The spaces $M_\omega^2(\rr d)$ are Hilbert spaces and $M^2(\rr d) = L^2(\rr d)$.

The $L^2$-inner product $(\cdot, \cdot)$ on $\cS(\rr d) \times \cS(\rr d)$ extends uniquely to a continuous 
sesquilinear form on $M_\omega^{p,q} (\rr d) \times M_{1/\omega}^{p',q'} (\rr d)$, and if $p,q < \infty$ then the dual space of 
$M_\omega^{p,q} (\rr d)$ may be identified with $M_{1/\omega}^{p',q'} (\rr d)$ via the form \cite[Theorem~11.3.6]{Grochenig1}.

We will also use modulation spaces with domain $\rr {2d}$ and weight functions defined on $\rr {4d}$ of the form
\begin{equation*}
\omega(x_1, x_2, \xi_1, \xi_2) = \eabs{x_1}^{\tau_1} \eabs{x_2}^{\tau_2} \eabs{\xi_1}^{\zeta_1} \eabs{\xi_2}^{\zeta_2}, \quad x_1, x_2, \xi_1, \xi_2 \in \rr d, 
\end{equation*}
with $\tau_1, \tau_2, \zeta_1, \zeta_2 \in \ro$. These spaces are denoted $M_{\tau_1,\tau_2,\zeta_1,\zeta_2}^{p,q}(\rr {2d})$ 
and will be used as symbols for Weyl pseudodifferential operators acting between modulation spaces $M_{t_1,s_1}^{p,q} (\rr d) \to M_{t_2,s_2}^{p,q} (\rr d)$. 
In fact according to \cite[Proposition~2.9]{Grochenig4} (cf. \cite[Theorem~4.2]{Toft1})
\begin{equation*}
a^w(x,D): M_{\omega_1}^{p,q} (\rr d) \to M_{\omega_2}^{p,q}(\rr d)
\end{equation*}
is continuous for all $p,q \in [1,\infty]$ if $a \in M_{\omega}^{\infty,1}(\rr {2d})$, 
$\omega_1, \omega_2 \in \mascP(\rr {2d})$, $\omega \in \mascP(\rr {4d})$,
and 
\begin{equation*}
\frac{\omega_2(X-Y)}{\omega_1(X+Y)}
\lesssim \omega(X, - 2 \J Y), \quad X, Y \in \rr {2d}, 
\end{equation*}
where
\begin{equation*}
\J =
\left(
\begin{array}{cc}
0 & I_d \\
-I_d & 0
\end{array}
\right) \in \rr {2d \times 2d}
\end{equation*}
is the matrix which plays a fundamental role in symplectic linear algebra \cite{Folland1,Grochenig1}.  

It follows that  
\begin{equation}\label{eq:contpsdomodsp}
a^w(x,D): M_{t_1,s_1}^{p,q} (\rr d) \to M_{t_2,s_2}^{p,q}(\rr d)
\end{equation}
is continuous for all $p,q \in [1,\infty]$
if $a \in M_{\tau_1,\tau_2,\zeta_1,\zeta_2}^{\infty,1}(\rr {2d})$ and
\begin{equation}\label{eq:weightsinequality1}
\eabs{ x - y }^{t_2} \eabs{ \xi - \eta }^{s_2} \eabs{ x + y }^{-t_1} \eabs{ \xi + \eta }^{-s_1}
\lesssim
\eabs{ x }^{\tau_1} \eabs{ \xi }^{\tau_2} \eabs{ \eta }^{\zeta_1} \eabs{ y }^{\zeta_2}, \quad 
x, \xi, y, \eta \in \rr d. 
\end{equation}
A particular case is $a \in M_{1 \otimes \omega}^{\infty,1}(\rr {2d})$
with $\omega(X) = \eabs{X}^r$ for $X \in \rr {2d}$ with $r \geqs 0$. 
Indeed $M_{1 \otimes \omega}^{\infty,1}(\rr {2d}) \subseteq M_{0,0,\frac{r}{2}, \frac{r}{2}}^{\infty,1}(\rr {2d})$
and hence we obtain from 
\eqref{eq:Peetre} and \eqref{eq:weightsinequality1} 
that 
\begin{equation}\label{eq:contpsdomodsp2}
a^w(x,D): M_{t,s}^{p,q} (\rr d) \to M_{t,s}^{p,q}(\rr d)
\end{equation}
is continuous for all $p,q \in [1,\infty]$, and all $t,s \in \ro$ such that 
$\max(|t|, |s|) \leqs \frac{r}{2}$. 

Consequentially $a^w(x,D)$ is continuous on $L^2(\rr d) = M^2(\rr d)$. 
Gr\"ochenig's spectral invariance theorem 
\cite[Theorem~4.6]{Grochenig3} says that if $a \in M_{1 \otimes \omega}^{\infty,1}(\rr {2d})$
and $a^w(x,D)$ is invertible on $L^2(\rr d)$, then $a^w(x,D)^{-1} = b^w(x,D)$ with 
$b \in M_{1 \otimes \omega}^{\infty,1}(\rr {2d})$. 
This so called Wiener property of $M_{1 \otimes \omega}^{\infty,1}(\rr {2d})$ is a refinement of Sj\"ostrand's original result \cite{Sjostrand2} which concerned the case 
when $r = 0$. 

Thus \eqref{eq:weightsinequality1} gives good mapping properties for operators with Weyl symbols in 
the space $M_{\tau_1,\tau_2,\zeta_1,\zeta_2}^{\infty,1}(\rr {2d})$
when $\tau_1, \tau_2, \zeta_1, \zeta_2 \geqs 0$, and furthermore the Wiener property holds.
But we will need also symbols $M_{0,\tau_2,\zeta_1,\zeta_2}^{\infty,1}(\rr {2d})$
with $\tau_2$ and $\zeta_2$ negative. 
If $\tau_1 = 0$ and $\tau_2, \zeta_2 < 0 \leqs \zeta_1$ then \eqref{eq:weightsinequality1} holds if 
\begin{equation}\label{eq:weightsinequality2}
\eabs{ x - y }^{t_2} \eabs{ x + y }^{-t_1} \eabs{ \xi - \eta }^{s_2}  \eabs{ \xi + \eta }^{-s_1}
\eabs{ \xi }^{-\tau_2} \eabs{ y }^{-\zeta_2}
\lesssim
\eabs{ \eta }^{\zeta_1}. 
\end{equation}
From \eqref{eq:Peetre} it follows that 
\begin{equation*}
\eabs{ \xi - \eta }^{s_2}  \eabs{ \xi + \eta }^{-s_1}
\eabs{ \xi }^{-\tau_2} \lesssim
\eabs{ \eta }^{\zeta_1}
\end{equation*}
provided $|s_1| + |s_2| \leqs \zeta_1$ and $s_2 \leqs s_1 + \tau_2$, and then \eqref{eq:weightsinequality2} 
reduces to 
\begin{equation}\label{eq:weightsinequality3}
\eabs{ x - y }^{t_2} \eabs{ x + y }^{-t_1} 
\eabs{ y }^{-\zeta_2} \lesssim 1. 
\end{equation}

Writing $y = \frac12(x + y) - \frac12(x-y)$ 
it follows again from \eqref{eq:Peetre} that \eqref{eq:weightsinequality3} is fulfilled provided
$t_2 \leqs \zeta_2$ and $t_1 \geqs - \zeta_2$. 
We summarize: If $a \in M_{0,\tau_2,\zeta_1,\zeta_2}^{\infty,1}(\rr {2d})$
with $\tau_2, \zeta_2 < 0 \leqs \zeta_1$ then the operator \eqref{eq:contpsdomodsp} is continuous for all $p,q \in [1,\infty]$ provided
\begin{equation}\label{eq:weightpowers}
\begin{aligned}
& t_1 \geqs - \zeta_2, \quad t_2 \leqs \zeta_2, \\
& s_2 \leqs s_1 + \tau_2, \quad |s_1| + |s_2| \leqs \zeta_1. 
\end{aligned}
\end{equation}

Finally we discuss Gabor frames \cite{Grochenig1,Wahlberg1} for $L^2(\rr {2d})$ defined by the Gaussian window function 
\begin{equation}\label{eq:gaussianwindow}
\Phi(X) = 2^d \pi^{d/2} \exp(-|X|^2), \quad X \in \rr {2d}. 
\end{equation}
Here we denote by 
\begin{equation}\label{eq:symplecticTFshift}
\Pi(X,Y) f(Z) = e^{2 i \sigma(Y,Z)} f(Z-X), \quad X,Y,Z \in \rr {2d}, 
\end{equation}
the composition of the translation operator and 
the \emph{symplectic} modulation operator $f \mapsto e^{2 i \sigma(Y,\cdot)} f$, defined using the symplectic form \cite{Folland1}
\begin{equation*}
\sigma( (x,\xi), (y,\eta) ) = \la y, \xi \ra - \la x, \eta \ra, \quad (x,\xi), (y,\eta) \in \rr {2d}.
\end{equation*}

If the parameters $a, b > 0$ satisfy $ab < \pi$ then $\{ \Pi(an,bk) \Phi \}_{n,k \in \zz {2d}}$ is a 
Gabor frame for $L^2(\rr {2d})$ which means that 
\begin{equation*}
A \| f \|_{L^2}^2 \leqs \sum_{{\bm \Lambda} \in \Theta} | \left(
f,\Pi({\bm \Lambda}) \Phi \right)|^2 \leqs B \|f \|_{L^2}^2 \quad
\forall f \in L^2(\rr {2d}), 
\end{equation*}
for some $0 < A \leqs B < \infty$. 
Here $\Theta = \{
(a n, b k) \}_{n,k \in \zz {2d}} \subseteq \rr {4d}$ is a lattice determined by $a, b > 0$. 
We denote elements in the lattice as 
\begin{equation}\label{eq:lattice}
{\bm \Lambda} = (\Lambda,\Lambda') \in \Theta, \quad \Lambda = a n, \quad \Lambda' =b k, \quad n, k \in \zz {2d}. 
\end{equation}

The Gabor frame operator $S = S(\Phi,\Theta)$, defined by
\begin{equation*}
Sf = \sum_{{\bm \Lambda} \in \Theta} ( f,\Pi({\bm \Lambda}) \Phi
) \, \Pi({\bm \Lambda}) \Phi,
\end{equation*}
is positive and invertible on $L^2(\rr {2d})$. For any $f \in
L^2(\rr {2d})$ we have a Gabor expansion
\begin{equation}\label{eq:gaborexp}
f = \sum_{{\bm \Lambda} \in \Theta} ( f,\Pi({\bm \Lambda} )
\widetilde{\Phi} ) \, \Pi({\bm \Lambda}) \Phi, \quad f \in
L^2(\rr {2d}),
\end{equation}
with unconditional convergence \cite{Grochenig1}. Here
$\widetilde{\Phi} = S^{-1} \Phi \in \cS(\rr {2d})$ \cite{Grochenig1,Janssen1} 
is the \emph{canonical dual window}.

Gabor theory has been generalized from the Hilbert space $L^2$ to
modulation spaces by Feichtinger, Gr\"ochenig and Leinert
\cite{Feichtinger1,Grochenig1,Grochenig2}. 
In fact if $\omega \in \mascP(\rr {4d})$ we have the norm equivalence
\begin{equation}\label{eq:gabornorm}
C^{-1} \| f \|_{M_\omega^{p,q}(\rr {2d})} 
\leqs \Big( \sum_{k \in \zz {2d}} \Big( \sum_{n \in \zz {2d}} | ( f,\Pi(a n,b k) \widetilde \Phi ) |^p \omega(a n, b k)^p \Big)^{\frac{q}{p}} \Big)^{\frac1q} \leqs
C\| f \|_{M_\omega^{p,q}(\rr {2d})}
\end{equation}
where $C > 0$, 
for the whole scale $1 \leqs p,q \leqs \infty$ of modulation spaces.
The expansion \eqref{eq:gaborexp} holds with unconditional convergence
if $p,q < \infty$, and in the weak$^*$ topology of
$M_{1/v}^\infty(\rr {2d})$ otherwise for some $v \in \mascP(\rr {4d})$. 

\subsection{Stochastic processes and generalized stochastic processes}\label{subsec:genstochproc}

Let $\Omega$ be a sample space equipped with a $\sigma$-algebra $\cB$ of subsets of $\Omega$
and let $\bP$ be a probability measure defined on $\cB$. 
The space of $\co$-valued random variables is the Hilbert space $L^2(\Omega)$ equipped with the inner product
\begin{equation*}
L^2(\Omega) \times L^2(\Omega) \ni (X,Y) \mapsto \bE ( X \overline Y ) = (X, Y )_{L^2(\Omega)}
\end{equation*}
where 
\begin{equation*}
\bE X = \int_{\Omega} X(\omega) \bP( \dd \omega)
\end{equation*}
is the expectation functional (integral). 
We write $X \perp Y$ if $\bE ( X \overline Y ) = 0$. 
The Hilbert subspace of $L^2(\Omega)$ of zero mean random variables is denoted $L_0^2(\Omega)$, and 
thus $\bE X = 0$ and $\bE |X|^2 < \infty$ for each element $X \in L_0^2(\Omega)$.

A second order zero mean stochastic process is a locally Bochner integrable map $f: \rr d \to L_0^2(\Omega)$. 
This space of stochastic processes is denoted $L_{\rm{loc}}^1( \rr d, L_0^2(\Omega) )$, 
and $C( \rr d, L_0^2(\Omega) ) \subseteq L_{\rm{loc}}^1( \rr d, L_0^2(\Omega) )$ denotes the subspace of 
continuous stochastic processes.
The cross-covariance function of $f,g \in L_{\rm{loc}}^1( \rr d, L_0^2(\Omega) )$ is
\begin{equation*}
k_{fg} ( x ,y) = \bE ( f(x)  \overline{g(y)} ), \quad x, y \in \rr d. 
\end{equation*}
The function $k_f = k_{f f}$ is the (auto-)covariance function of $f$. 
By the Cauchy--Schwarz inequality in $L^2(\Omega)$ we have 
\begin{equation*}
|k_{f g} ( x ,y)|^2 
\leqs k_f ( x ,x) \, k_g ( y ,y), \quad x, y \in \rr d,
\end{equation*}
which implies that $k_{f g}$ extends to an element in $\cD'(\rr {2d})$ \cite{Gelfand4} when $f,g \in L_{\rm{loc}}^1( \rr d, L_0^2(\Omega) )$. 
Defining the cross-covariance operator 
\begin{equation*}
( \cK_{f g} \fy, \psi) = ( k_{f g}, \psi \otimes \overline \fy )_{L^2(\rr {2d} )}, \quad \fy, \psi \in C_c^\infty(\rr d), 
\end{equation*}
yields a continuous linear cross-covariance operator $\cK_{f g}: C_c^\infty(\rr d) \to \cD'(\rr d)$. 

Since Fubini's theorem gives
\begin{align*}
( k_f, \fy \otimes \overline \fy )_{L^2(\rr {2d})}
& = \bE \left| \int_{\rr d}   f(x) \overline{\fy (x)} \, \dd x \right|^2
\geqs 0 \quad \forall \fy \in C_c^\infty(\rr d),    
\end{align*}
it follows that $k_f$ is the kernel of a non-negative linear continuous covariance operator $\cK_f = \cK_{f f}: C_c^\infty(\rr d) \to \cD'(\rr d)$.

We use Gelfand and Vilenkin's concept of generalized stochastic process defined as follows \cite{Gelfand4}, cf. \cite{Feichtinger2,Hormann1,Keville1}. 

\begin{defn}\label{def:gsp}
A second order zero mean generalized stochastic process (GSP) $u \in \cL ( C_c^\infty(\rr d), L_0^2(\Omega) )$ is a 
conjugate linear continuous operator $u: C_c^\infty(\rr d) \to L_0^2(\Omega)$, 
written as $(u,\fy) \in L_0^2(\Omega)$ for $\fy \in C_c^\infty(\rr d)$. 
\end{defn}

Thus for each $K \Subset \rr d$ there exist $C > 0$ and $k \in \no$ 
such that  
\begin{equation*}
\| (u,\fy) \|_{L_0^2(\Omega)}
\leqs C \sum_{|\alpha| \leqs k} \sup_{x \in \rr d} |\pd \alpha \fy (x)|, \quad \fy \in C_c^\infty(K).
\end{equation*}

As in ordinary distribution theory \cite{Gelfand1,Hormander1} a GSP is always differentiable as
\begin{equation*}
( D^\alpha u,\fy) = (u, D^\alpha \fy), \quad \fy \in C_c^\infty (\rr d), \quad \alpha \in \nn d. 
\end{equation*}
For $u \in \cL ( C_c^\infty(\rr d), L_0^2(\Omega) )$ we denote 
by $L_u^2(\Omega) \subseteq L_0^2(\Omega)$ the \emph{time domain} \cite{Keville1} of $u$, i.e.
the linear subspace of the closure of its image:
\begin{equation}\label{eq:timedomain}
L_u^2(\Omega) 
= {\rm closure} \, \{ (u,\fy), \ \fy \in C_c^\infty(\rr d)  \} \subseteq L_0^2(\Omega)
\end{equation}

Let $u,v \in  \cL ( C_c^\infty(\rr d), L_0^2(\Omega) )$. The cross-covariance distribution is defined by 
\begin{equation}\label{eq:crosscovdist}
(k_{u v}, \fy \otimes \overline \psi ) 
= \bE ( (u,\fy)  \overline{ (v,\psi)} ), \quad \fy, \psi \in C_c^\infty (\rr d). 
\end{equation}
For each pair $K_1, K_2 \Subset \rr d$
there is $C > 0$ and $k_1, k_2 \in \no$ such that 
\begin{align*}
|(k_{u v}, \fy \otimes \overline \psi )|
& \leqs C \sum_{|\alpha| \leqs k_1} \sup_{x \in \rr d} |\pd \alpha \fy (x)| \, 
\sum_{|\beta| \leqs k_2} \sup_{x \in \rr d} |\pd \beta \psi (x)|, \\
& \qquad \qquad \fy \in C_c^\infty(K_1), \quad \psi \in C_c^\infty(K_2). 
\end{align*}
Thus $k_{u v}$ is a sesquilinear (conjugate linear in the first argument) continuous form on $C_c^\infty(\rr d) \times C_c^\infty(\rr d)$. 
If we equip $\cD'(\rr d)$ with its weak$^*$ topology then 
\begin{equation*}
( \cK_{u v} \psi, \fy ) = (k_{u v}, \fy \otimes \overline \psi ), \quad \fy, \psi \in C_c^\infty(\rr d), 
\end{equation*}
defines a linear continuous operator $\cK_{u v}: C_c^\infty(\rr d) \to \cD' (\rr d)$, 
called the cross-covariance operator. 
Note that $(\cK_{u v} \psi, \fy ) = \overline{ (\cK_{v u} \fy, \psi ) }$. 
From the Schwartz kernel theorem \cite[Theorem~5.2.1]{Hormander1} it follows that $k_{u v} \in \cD'(\rr {2d})$. 
If $k_{u v} \equiv 0$ then $\cK_{u v} = 0$ and we say that $u$ and $v$ are uncorrelated. 
This means that 
\begin{equation*}
L_u^2(\Omega) \perp L_v^2(\Omega). 
\end{equation*}

If $v = u$ then we call $k_u = k_{u u} \in \cD'(\rr {2d})$ the (auto-)covariance distribution
and $\cK_u = \cK_{u u} \in \cL( C_c^\infty(\rr d), \cD' (\rr d) )$ the (auto-)covariance operator of $u$. 
Then $(k_u, \fy \otimes \overline \fy ) \geqs 0$ for all $\fy \in C_c^\infty(\rr d)$
so $\cK_u \geqs 0$ is non-negative in the sense of 
\begin{equation*}
( \cK_u \fy, \fy ) \geqs 0 \quad \forall \fy \in C_c^\infty(\rr d). 
\end{equation*}

\begin{rem}\label{rem:existenceGSP}
In \cite[Chapter~3 \S 2.3]{Gelfand4} and \cite[proof of Proposition~8.1]{Wahlberg2} it is shown that any sesquilinear continuous form $k$ on 
$C_c^\infty(\rr d) \times C_c^\infty(\rr d)$ which gives rise to a 
non-negative continuous operator $C_c^\infty(\rr d) \to \cD' (\rr d)$ 
is the auto-covariance distribution of a Gaussian GSP
$u \in \cL ( C_c^\infty(\rr d) , L_0^2(\Omega) )$ such that 
\begin{align*}
\bE \left( (u, \fy) \overline{(u, \psi)} \right) & = (k, \fy \otimes \overline \psi ), \\
\bE \left( (u, \fy) (u, \psi) \right) & \equiv 0, \quad \fy, \psi \in C_c^\infty(\rr d). 
\end{align*}
\end{rem}

\begin{example}\label{ex:example1}
If $p > 0$ there exists a GSP $u \in \cL ( C_c^\infty(\rr d) , L_0^2(\Omega) )$
such that $(k_u, \fy \otimes \overline \psi) = p (\psi,\fy)_{L^2}$. 
This is a consequence of \cite[Chapter~3 \S 2.3]{Gelfand4}. 
Then the covariance operator equals a positive multiple of the identity: $\cK_u = p I$, 
and $k_u (x,y) = p \delta_0(x-y)$. 
This GSP is called white noise with power $p$, and is an example of a
GSP which is not a stochastic process. 
\end{example}

\begin{example}\label{ex:example2}
Let $p > 0$, let $u$ be a white noise GSP with power $p$, and let $\alpha \in \nn d$.
Then $D^\alpha u$ is a GSP with covariance distribution
\begin{align*}
(k_{D^\alpha u}, \fy \otimes \overline \psi ) 
& = \bE ( (u, D^\alpha \fy)  \overline{ (u, D^\alpha \psi)} )
= ( k_u, D^\alpha \fy \otimes \overline{ D^\alpha \psi } ) \\
& = p ( D^\alpha \psi, D^\alpha \fy )_{L^2}
= p ( D^{2 \alpha} \psi, \fy)_{L^2}, 
 \quad \fy, \psi \in C_c^\infty (\rr d), 
\end{align*}
where we integrate by parts. 
Thus $\cK_{D^\alpha u} = p D^{2 \alpha}$,
and the covariance distribution is 
\begin{equation*}
k_{D^\alpha u} =p  \left( D^{\alpha} \otimes (-D)^{\alpha} \right) \left( (1 \otimes \delta_0) \circ T^{-1} \right)
\end{equation*}
with the matrix (cf. \eqref{eq:T})
\begin{equation*}
T^{-1} = 
\left(
\begin{array}{ll}
\frac12 I_d & \frac12 I_d \\
I_d & - I_d
\end{array}
\right) \in \rr {2d \times 2d}. 
\end{equation*}
In fact $k_u = p (1 \otimes \delta_0) \circ T^{-1}$ according to Example \ref{ex:example1}. 
\end{example}

A stochastic process $f \in L_{\rm{loc}}^1 (\rr d, L_0^2(\Omega) )$ may be considered a generalized stochastic process in $\cL ( C_c^\infty(\rr d) , L_0^2(\Omega) )$ by means of 
\begin{equation*}
C_c^\infty (\rr d) \ni \fy \mapsto (f, \fy) = \int_{\rr d} f(x ) \, \overline{ \fy (x) } \dd x, 
\end{equation*}
and then there is consistency between the covariance function and the covariance distribution as
\begin{equation*}
( k_f, \psi \otimes \overline \fy )_{L^2(\rr {2d})}
= \iint_{\rr {2d}} \bE \left( f(x) \overline{ f(y)}  \right) 
\overline{\psi(x)} \fy (y) \dd x \, \dd y
= \bE \left( (f, \psi) \overline{ (f, \fy)} \right). 
\end{equation*}

\begin{rem}\label{rem:Fouriertransform}
As in ordinary distribution theory \cite{Gelfand1,Hormander1} a generalized stochastic 
process may extend to the domain $\cS(\rr d) \supseteq C_c^\infty(\rr d)$ and is then called \emph{tempered}. 
If $u$ is a tempered GSP then $u \in \cL ( \cS(\rr d), L_0^2(\Omega) ) \subseteq \cL ( C_c^\infty(\rr d), L_0^2(\Omega) )$. 
The Fourier transform of $u$ can then be defined as $(\wh u, \wh \fy) = (u, \fy)$ for $\fy \in \cS(\rr d)$.
\end{rem}

If $u \in \cL ( \cS(\rr d), L_0^2(\Omega) )$ is a tempered $\gsp$ then the covariance operator is continuous 
$\cK_u: \cS(\rr d) \to \cS'(\rr d)$ and non-negative. Its Schwartz kernel $k_u \in \cS'(\rr {2d})$ is a tempered distribution, and 
its Weyl symbol is denoted $a_u \in \cS'(\rr {2d})$. The distributions $k_u, a_u \in \cS'(\rr {2d})$ are connected by 
\eqref{eq:KernelWeylsymbol}.

\section{Filtering and the equation for the optimal filter}\label{sec:optimalfilter}

Suppose a stochastic process $g$ is a noisy observation of a message
stochastic process $f$. 
A common assumption is the additive model 
\begin{equation*}
g = f + w, 
\end{equation*}
where $w$
is a noise stochastic process which is uncorrelated with $f$, i.e. $\bE( f(x) 
\overline{w(y)}) = 0$ for all $x,y \in \rr d$. 
To recover $f$ approximately from $g$ we may try to 
filter $g$ linearly using a linear operator $F$ with kernel function $k$ to get an optimal
approximation of $f$, denoted $f_o$, \cite{Fomin1,Wiener1} as
\begin{equation}\label{eq:estimatorsp}
f_o(x) = (F g)(x) = \int_{\rr d} k(x,y) g(y) \dd y.
\end{equation}
This integral makes sense e.g. if we assume $k \in L^2(\rr {2d})$ and $g \in L^2(\rr d, L_0^2(\Omega))$. 

The optimality of the approximation $f_o$ refers to the postulate to pick 
a filter kernel $k$ that minimizes the mean
square error 
\begin{equation*}
\bE | f(x) - f_o(x)|^2 = \| f(x) - f_o(x) \|_{L^2(\Omega)}^2
\end{equation*}
for all $x \in \rr d$. 

The idea of filtering can be extended from stochastic processes to
generalized stochastic processes as follows. Let $u,v  \in \cL ( C_c^\infty(\rr d) , L_0^2(\Omega) )$
have auto-covariance distributions $k_u, k_v \in \cD'(\rr {2d})$,
respectively, and cross-covariance distribution $k_{u v} \in \cD'(\rr {2d})$. 
These are sesquilinear continuous forms on $C_c^\infty (\rr d) \times  C_c^\infty (\rr d)$. 

Each of these kernels gives rise to continuous linear (auto-, cross-)covariance operators 
denoted $\cK_u, \cK_v, \cK_{u v}: C_c^\infty (\rr d) \to \cD'(\rr d)$ respectively.  
Then $\cK_u \geqs 0$ and $\cK_v \geqs 0$ on $C_c^\infty(\rr d)$. 

We will assume that
\begin{equation}\label{eq:filterassumption1}
F: C_c^\infty(\rr d) \to C_c^\infty(\rr d)
\end{equation}
is a linear continuous operator, called \emph{filter}. 
The adjoint $F^*$ defined by
\begin{equation}\label{eq:adjoint}
( F \fy, \psi) = ( \fy, F^* \psi), \quad \fy, \psi \in C_c^\infty(\rr d), 
\end{equation}
is then a continuous linear operator $C_c^\infty(\rr d) \to \cD' (\rr d)$. 
We assume continuity of the adjoint:
\begin{equation}\label{eq:filterassumption2}
F^*: C_c^\infty(\rr d) \to C_c^\infty(\rr d).
\end{equation}
Replacing $\fy \in C_c^\infty(\rr d)$ with $\fy \in \cD'(\rr d)$, the formula \eqref{eq:adjoint} extends $F$ uniquely to a linear continuous operator on $\cD'(\rr d)$, equipped with the weak$^*$ topology. 

Given a linear operator $F$ that satisfy \eqref{eq:filterassumption1}, \eqref{eq:filterassumption2} and $v \in \cL ( C_c^\infty(\rr d) , L_0^2(\Omega) )$ we define the filtered generalized stochastic process $u_o \in \cL ( C_c^\infty(\rr d) , L_0^2(\Omega) )$ as $u_o = F v$, that is
\begin{equation}\label{eq:estimatorgsp}
(u_o, \fy ) = (F v, \fy ) = ( v, F^* \fy), \quad \fy \in C_c^\infty(\rr d),
\end{equation}
which extends \eqref{eq:estimatorsp} from stochastic processes $g \in L_{\rm{loc}}^1(\rr d, L_0^2(\Omega))$ to $v \in \cL ( C_c^\infty(\rr d) , L_0^2(\Omega) )$. 
To wit, if $F$ has integral kernel $k$ and $v \in L_{\rm{loc}}^1(\rr d, L_0^2(\Omega))$ then \eqref{eq:estimatorgsp} reduces to
\begin{equation*}
\iint_{\rr {2d}} k(x,y) v(y) \overline{\fy(x)} \, \dd x \, \dd y.
\end{equation*}

Let $u,v \in \cL ( C_c^\infty(\rr d) , L_0^2(\Omega) )$. We assume that we have access to $v$ but not to $u$. 
The $\gsp$ $v$ is assumed to be a noise corrupted version of $u$, and we want to recover the latter with optimally small error
using a filter $F$ as in \eqref{eq:estimatorgsp}. 

For a fixed arbitrary $\fy \in C_c^\infty(\rr d) \setminus \{ 0 \}$ we may derive an
equation for the filter operator $F$ that is optimal in the sense
of minimizing the mean square error 
\begin{equation*}
\bE | ( u ,\fy) - ( u_o,\fy ) |^2. 
\end{equation*}
By \eqref{eq:estimatorgsp} we have $(u_o,\fy) \in L_v^2(\Omega)$ for any filter linear operator $F$
that satisfies \eqref{eq:filterassumption1} and \eqref{eq:filterassumption2}.
We may formulate optimality as follows.

\begin{prop}\label{prop:ortogonal}
Let $u,v \in \cL ( C_c^\infty(\rr d) , L_0^2(\Omega) )$, let $\fy \in C_c^\infty(\rr d) \setminus \{ 0 \}$ be fixed, 
and suppose $F$ is a linear operator that satisfy \eqref{eq:filterassumption1} and \eqref{eq:filterassumption2}. 
The filter $F$ in \eqref{eq:estimatorgsp} is optimal for $\fy$ if and only if
\begin{equation}\label{eq:optimalortogonal}
(u - u_o,\fy) \perp L_v^2(\Omega). 
\end{equation}
\end{prop}

\begin{proof}
We may uniquely decompose $(u,\fy) = X_0 + X_1$ where $X_0 \in L_v^2(\Omega)$ and $X_1 \in L_v^2(\Omega)^\perp$. 
By elementary Hilbert space theory $X_0$ is the unique optimal vector
in $L_v^2(\Omega)$ that satisfies
\begin{equation}\label{eq:bestapprox}
X_0 = {\arg \inf}_{X \in L_v^2(\Omega)} \| (u,\fy) - X \|_{L^2(\Omega)},
\end{equation}
and $\| X_1 \|_{L^2(\Omega)} = \inf_{X \in L_v^2(\Omega)} \| (u,\fy) - X
\|_{L^2(\Omega)}$. 

Suppose $(u - u_o,\fy) \perp L_v^2(\Omega)$. 
Then $(u,\fy) = (u_o,\fy) + Y_1$ where $Y_1 \in L_v^2(\Omega)^\perp$. 
By \eqref{eq:estimatorgsp} we have $(u_o,\fy) \in L_v^2(\Omega)$, and it follows from the uniqueness of the decomposition
$(u,\fy) = X_0 + X_1$ that $Y_1 = X_1$ and $(u_o,\fy)  = X_0$. 
By \eqref{eq:bestapprox} it thus follows that $(u_o,\fy)$ is optimal.

On the other hand, suppose that $(u - u_o,\fy) \notin L_v^2(\Omega)^\perp$, 
i.e. $(u,\fy) = (u_o,\fy)  + Y_0 + Y_1$ where $Y_0 \in
L_v^2(\Omega)$, $Y_1 \in L_v^2(\Omega)^\perp$ and $Y_0 \neq 0$. Again by the uniqueness
of the decomposition $(u,\fy) = X_0 + X_1$ we have
$(u_o,\fy) = X_0 - Y_0$. 
There exists $\psi \in C_c^\infty(\rr d) \setminus 0$ such that 
$\| Y_0 - (v,\psi) \|_{L^2(\Omega)}^2 < \| Y_0 \|_{L^2(\Omega)}^2$. 

If $F_0^*$ is the linear continuous operator on $C_c^\infty(\rr d)$ 
\begin{equation*}
F_0^* g = \| \fy \|_{L^2}^{-2} \left(g, \fy \right) \psi, \quad g \in C_c^\infty(\rr d), 
\end{equation*}
then 
\begin{equation*}
F_0 g = \| \fy \|_{L^2}^{-2} (g,\psi) \fy, \quad g \in C_c^\infty(\rr d). 
\end{equation*}
Thus $F_0$ and $F_0^*$ are both continuous on $C_c^\infty (\rr d)$, and $F_0^* \fy = \psi$. 
We have
\begin{align*}
\| (u,\fy)  - (v, (F + F_0)^* \fy \|_{L^2(\Omega)}^2 
& = \| X_1 + ( Y_0 - ( v,
\psi ) \|_{L^2(\Omega)}^2 \\
& = \| X_1 \|_{L^2(\Omega)}^2 + \| Y_0 - ( v, \psi ) \|_{L^2(\Omega)}^2 \\
&  < \| X_1 \|_{L^2(\Omega)}^2 + \| Y_0 \|_{L^2(\Omega)}^2 
= \| ( u - u_o, \fy) \|_{L^2(\Omega)}^2,
\end{align*}
which means that the filter $F$ in $(u_o,\fy)$ is not
optimal.
\end{proof}

Condition \eqref{eq:optimalortogonal} generalized to all $\fy \in C_c^\infty(\rr d)$
can be expressed with \eqref{eq:crosscovdist}
as
\begin{equation*}
(k_{u v}, \fy \otimes \overline{\psi} ) 
= (k_v, F^* \fy \otimes \overline{\psi} ),
 \quad \forall \fy, \psi \in C_c^\infty(\rr d).
\end{equation*}
Thus 
\begin{equation*}
( \cK_{u v} \psi, \fy ) 
= ( \cK_v \psi, F^* \fy )
= (F \cK_v \psi, \fy )
 \quad \forall \fy, \psi \in C_c^\infty(\rr d), 
\end{equation*}
which means that
\begin{equation}\label{eq:operatoreq1}
\cK_{u v} = F \cK_v 
\end{equation}
as operators $C_c^\infty(\rr d) \to \cD'(\rr d)$. 
Given the operators $\cK_{u v}, \cK_v \in  \cL( C_c^\infty(\rr d), \cD' (\rr d) )$, this is an operator equation for the optimal linear filter $F$, 
which as noted above may be considered a continuous linear operator $F: \cD' (\rr d) \to \cD' (\rr d)$. 

\begin{rem}\label{rem:OpEq}
In \cite{Wahlberg1} the operator equation \eqref{eq:operatoreq1} is deduced in a different functional framework. 
In fact we use $\cL ( M^1(\rr d) , L_0^2(\Omega) )$ as the class of $\gsp$s, 
and the filters are pseudodifferential operators with symbols in modulation spaces. 
The space $\cL ( M^1(\rr d) , L_0^2(\Omega) )$ of $\gsp$s is studied more carefully in \cite{Hormann1,Feichtinger2,Keville1}.
The test function space is Feichtinger's algebra $M^1(\rr d)$ which is a Fourier invariant Banach space of continuous integrable functions with integrable Fourier transform,
and $\cS(\rr d) \subseteq M^1(\rr d)$ is an embedding. 
This gives a smaller space of $\gsp$s than the space of tempered $\gsp$s: $\cL ( M^1(\rr d) , L_0^2(\Omega) ) \subseteq \cL ( \cS(\rr d) , L_0^2(\Omega) )$. 
If $u \in \cL ( M^1(\rr d) , L_0^2(\Omega) )$ then the covariance operator is continuous $\cK_u: M^1(\rr d) \to M^\infty(\rr d)$ with 
$M^\infty(\rr d)$ equipped with its weak$^*$ topology. 
\end{rem}

By the Pythagorean theorem in $L^2(\Omega)$ the optimal (minimal)
mean square error for $\fy \in C_c^\infty(\rr d)$ is
\begin{equation}\label{eq:minmse1}
\begin{aligned}
J(\fy) & = \bE | (u - u_o,\fy) |^2 
= \bE | (u,\fy) |^2 - \bE | (u_o,\fy) |^2 
= \bE | (u,\fy) |^2 - \bE | (v, F^* \fy) |^2 \\
& = ( k_u, \fy \otimes \overline \fy ) - ( k_v, F^* \fy \otimes \overline{ F^* \fy }) \\
& = ( ( \cK_u - F \cK_v F^*)\fy, \fy) 
= ( ( \cK_u - \cK_{u v} F^*)\fy, \fy). 
\end{aligned}
\end{equation}

Suppose now the more specific model, common in engineering applications, 
\begin{equation}\label{eq:signalplusnoise}
v = u + w
\end{equation}
where $u, w \in  \cL ( C_c^\infty(\rr d) , L_0^2(\Omega) )$, 
$u$ is considered a signal and $w$ is considered as noise uncorrelated to $u$, that is $k_{u w} \equiv 0$. 
Then $\cK_{u v} = \cK_u$ and $\cK_v = \cK_u + \cK_w$, since $k_v = k_u + k_w$. 
The operator equation \eqref{eq:operatoreq1} is then
\begin{equation}\label{eq:operatoreq2}
\cK_u = F ( \cK_u +  \cK_w)
\end{equation}
in the cone of non-negative continuous linear operators $C_c^\infty(\rr d) \to \cD'(\rr d)$. 
The minimal mean square error for $\fy \in C_c^\infty(\rr d)$ is obtained from \eqref{eq:minmse1} and \eqref{eq:operatoreq2} as
\begin{equation}\label{eq:minmse2}
J(\fy) = \overline{J(\fy)}
= \overline{ (  \cK_u (I - F^* ) \fy, \fy) }
= ( \cK_u \fy, (I - F^* ) \fy) 
= ( (I - F ) \cK_u \fy, \fy)
= ( F \cK_w \fy, \fy).
\end{equation}

As a tool to solve \eqref{eq:operatoreq2} we will use

\begin{lem}\label{lem:operatorequation}
If $T: \cS(\rr d) \to \cS'(\rr d)$ is a linear continuous operator 
then 
$T = 0$ if and only if $(T f,f) = 0$ for all $f \in \cS(\rr d)$. 
\end{lem}

\begin{proof}
The claim is an immediate consequence of the polarization identity
\begin{equation*}
( T(f+g), f+g ) - ( T(f-g), f-g ) + i \Big( ( T(f+ig), f+ig ) - ( T(f-ig), f-ig ) \Big) = 4 ( T f,g).
\end{equation*}
\end{proof}

\section{The optimal filter for wide-sense stationary generalized stochastic processes}
\label{sec:optimalwss}

\subsection{Wide-sense stationary $\gsp$s}
\label{subsec:WSSgsp}

A zero mean stochastic process $f$ is said to be wide-sense stationary (WSS) if its covariance function satisfies 
$k_f (x,y) = \kappa_f(x-y)$ for a function $\kappa_f: \rr d \to \co$. 
This means that the stochastic process is second order translation invariant. 
The function $\kappa_f$ is non-negative definite in the following sense:
\begin{equation*}
\sum_{j,k=1}^n \kappa_f (x_j-x_k) z_j \overline{z}_k \geqs 0 \quad \forall \{ x_j \}_{j=1}^n \subseteq \rr d, \ \{ z_j \}_{j=1}^n \subseteq \co, \quad n \in \no \setminus 0. 
\end{equation*}

Bochner's theorem \cite{Katznelson1,Rudin1} says that a function $\kappa: \rr d \to \co$ is continuous and non-negative definite
if and only if $\mu = \cF \kappa$ is a non-negative bounded Radon measure on $\rr d$. 

Let $u \in \cL ( C_c^\infty(\rr d) , L_0^2(\Omega) )$ be a $\gsp$. 
Then $k_u \in \cD'(\rr {2d})$ as explained in Section \ref{subsec:genstochproc}. 
We call $u$ WSS if its covariance distribution $k_u$ is translation invariant: 
\begin{equation*}
(k_u, T_x \fy \otimes T_x \overline{\psi}) = (k_u, \fy \otimes \overline{\psi}) \quad \forall \fy, \psi \in C_c^\infty (\rr d) \quad \forall x \in \rr d, 
\end{equation*}
cf. \cite[Chapter 3, \S3]{Gelfand4}.
This means that the covariance operator is translation invariant as $T_x \cK_u = \cK_u T_x$
for all $x \in \rr d$. 
Under the assumption WSS there exists $\kappa_u \in \cD'(\rr d)$ such that 
\begin{equation*}
(k_u, \fy \otimes \overline{\psi})
= (\kappa_u, \fy * \psi^* ) \quad \forall \fy, \psi \in C_c^\infty (\rr d), 
\end{equation*}
where $\psi^*(x) = \overline{\psi (-x)}$, 
see 
\cite[Chapter 2, \S3.5]{Gelfand4} and \cite[Theorem~3.1.4$'$]{Hormander1}. 
By the Bochner--Schwartz theorem \cite[Chapter 2, \S3.3, Theorem~3]{Gelfand4} there exists a spectral non-negative tempered Radon measure $\mu_u = (2 \pi)^{\frac{d}{2}} \wh \kappa_u \in \cM_+ (\rr d)$ that satisfies \eqref{eq:temperedmeasure} for some $s \geqs 0$. 
Hence (cf. \eqref{eq:convolutionFourier})
\begin{equation*}
( \kappa_u, \fy * \psi^* ) = \int_{\rr d}  \wh \psi (\xi) \, \overline{ \wh \fy (\xi) } \, \dd \mu_u (\xi). 
\end{equation*}
This gives for $\fy, \psi \in C_c^\infty (\rr d)$
\begin{equation}\label{eq:gspwss1}
\bE \left( (u, \fy ) \overline{(u,\psi)} \right)
= (k_u, \fy \otimes \overline{\psi})
= ( \cK_u \psi, \fy)
= ( \kappa_u, \fy * \psi^* ) 
= ( \psi, \fy )_{\cF L^2(\mu_u)}
\end{equation}
and in particular
\begin{equation}\label{eq:gspwss2}
\| (u, \fy ) \|_{L^2(\Omega)}^2
= \int_{\rr d}  |\wh \fy (\xi) |^2 \, \dd \mu_u (\xi). 
\end{equation}

Let $u \in \cL ( C_c^\infty(\rr d) , L_0^2(\Omega) )$ be WSS and denote by $\mu_u \in \cM_+ (\rr d)$ the corresponding spectral non-negative tempered Radon measure. 
The identity \eqref{eq:gspwss1} implies that 
the linear map 
$C_c^\infty(\rr d) \ni \fy \mapsto \overline{(u,\fy) }$ 
extends uniquely to 
a unitary operator between Hilbert spaces $\cF L^2(\mu_u) \to L_u^2(\Omega)$ (cf. \eqref{eq:timedomain}), 
and $\cK_u$ extends uniquely to the identity operator on $\cF L^2(\mu_u)$.
Alternatively we may regard $\cK_u$ as the convolution operator
\begin{equation}\label{eq:WSSconvop}
\cK_u f = \kappa_u * f
\end{equation}
and we may extend the domain to $f \in \cS(\rr d)$.
(Note that $\kappa_u \in \cS'(\rr d)$.)
This yields a continuous operator 
\begin{equation}\label{eq:WSSconvopcont}
\cK_u: \cS(\rr d) \to \left( C^\infty \cap \cS' \right) (\rr d). 
\end{equation}

\begin{rem}\label{rem:wssexistence}
For any given $\mu \in \cM_+ (\rr d)$ there exists a WSS $\gsp$ $u$ such that $\mu_u = \mu$. 
This is a consequence of Remark \ref{rem:existenceGSP}.
\end{rem}

\begin{rem}\label{rem:FouriertransformWSS}
If $u$ is a WSS $\gsp$ then from Remark \ref{rem:Fouriertransform} and $\cS(\rr d) \subseteq \cF L^2(\mu_u)$ it follows that $u$ is tempered, 
possess a Fourier transform $\wh u: L^2(\mu_u) \to L_u^2(\Omega)$
such that $\fy \mapsto \overline{( \wh u,\fy) }$
is unitary, and $\supp \wh u = \supp \mu_u$.
\end{rem}

\begin{rem}\label{rem:derivativeWSS}
If $u$ is a WSS $\gsp$ and $\alpha \in \nn d$ then \eqref{eq:gspwss1} gives
\begin{equation*}
\left( k_{D^\alpha u}, \fy \otimes \overline{\psi} \right)
= (k_u, D^\alpha \fy \otimes \overline{D^\alpha \psi})
= \int_{\rr d}  \wh \psi (\xi) \, \overline{ \wh \fy (\xi) } \xi^{2 \alpha } \, \dd \mu_u (\xi). 
\end{equation*}
Thus $D^\alpha u$ is WSS and its spectral measure is $\dd \mu_{D^\alpha u} = \xi^{2 \alpha } \, \dd \mu_u$.
\end{rem}

\begin{rem}\label{rem:feichtingeralgebra}
Consider the framework $\cL ( M^1(\rr d) , L_0^2(\Omega) )$ with test functions in Feichtinger's algebra, 
cf. Remark \ref{rem:OpEq}. 
If $u \in \cL ( M^1(\rr d) , L_0^2(\Omega) )$ is WSS then its spectral non-negative measure $\mu_u$
is translation bounded \cite[Corollary~6]{Feichtinger2}
which means that 
\begin{equation*}
\sup_{x \in \rr d} |( \mu_u, T_x \fy)| < \infty \quad \forall \fy \in C_c(\rr d). 
\end{equation*}
\end{rem}

\begin{rem}\label{rem:WSSboundedmeasure}
If $s = 0$ in \eqref{eq:temperedmeasure} then the measure $\mu_u \in \cM_+(\rr d)$ is bounded. 
This implies that 
\begin{equation}\label{eq:boundedspectralmeas}
\kappa_u(x) = ( 2 \pi )^{-d} \int_{\rr d} e^{i \la x, \xi \ra } \dd \mu_u (\xi)
= (2 \pi)^{-\frac{d}{2}} \cF^{-1} \mu_u (x)
\end{equation}
is a function in $(C \cap L^\infty )(\rr d)$.
In turn this means that the initial assumption $u \in  \cL ( \cS(\rr d) , L_0^2(\Omega) )$
can be strengthened to $u \in C(\rr d, L_0^2(\Omega))$, that is $u$ is a continuous 
stochastic process. 
\end{rem}

\begin{rem}\label{rem:WSSabscontmeas}
If the measure $\mu_u \in \cM_+(\rr d)$ is absolutely continuous with respect to Lebesgue measure
then we may write $\dd \mu_u(\xi) = f_u(\xi) \dd \xi$ where $f_u \in L_{\rm{loc}}^1(\rr d)$ \cite{Folland1}.
If further $f_u \in L^1(\rr d)$ then $\mu_u \in \cM_+(\rr d)$ is finite. 
By \eqref{eq:boundedspectralmeas} $\kappa_u = (2 \pi)^{-\frac{d}{2}} \cF^{-1} f_u \in C_0(\rr d)$.
\end{rem}

Suppose $u$ is a WSS $\gsp$ with spectral measure $\mu_u \in \cM_+ (\rr d)$, 
and suppose $\mu \in \cM_+ (\rr d)$ satisfies $\mu \geqs \mu_u$. 
From \eqref{eq:gspwss1} and the Cauchy--Schwarz inequality in $L^2(\mu)$ we get
\begin{equation}\label{eq:bilinearbound}
|( k_u ,\fy \otimes \overline \psi )|
= |( \cK_u \psi, \fy)|
= |( \psi, \fy)_{\cF L^2(\mu_u)}|
\leqs \| \psi \|_{\cF L^2(\mu)} \,\| \fy \|_{\cF L^2(\mu)}. 
\end{equation}
Thus $k_u$ extends to a sesquilinear form on $\cF L^2(\mu) \times \cF L^2(\mu)$. 
From \cite[Corollary~II.2]{Reed1} we get the following conclusion. 
There exists a unique $\cT_{u,\mu} \in \cL( \cF L^2(\mu) )$ such that $\| \cT_{u,\mu} \|_{ \cL( \cF L^2(\mu) ) } \leqs 1$ and 
\begin{equation}\label{eq:bilin2lin}
( \cK_u \psi, \fy)
= ( \cT_{u,\mu} \psi, \fy)_{\cF L^2(\mu) }, \quad \psi, \fy \in \cF L^2(\mu). 
\end{equation}
If $\mu = \mu_u$ then $\cT_{u,\mu} = \cK_u = \id_{\cF L^2(\mu_u)}$ as already observed. 

\begin{example}\label{ex:example1b}
For the white noise GSP $u$ in Example \ref{ex:example1}
we have $(k_u, \fy \otimes \overline \psi) = p (\psi,\fy)_{L^2} = p (\wh \psi, \wh \fy)_{L^2}$
by Plancherel's theorem. The corresponding spectral tempered measure is therefore a positive multiple of 
Lebesgue measure
$\dd \mu_u = p \, \dd \xi$ on $\rr d$. 
\end{example}

\begin{example}\label{ex:example2b}
Let $u$ be white noise with power $p > 0$ as in Example \ref{ex:example1} and let $\alpha \in \nn d$. 
Remark \ref{rem:derivativeWSS} combined with Example \ref{ex:example1b} gives 
$\mu_{D^\alpha u} = p \, \xi^{2 \alpha} \dd \xi$, and 
the covariance operator is $\cK_{D^\alpha u} = p D^{2 \alpha}$. 
This example reveals that although the covariance operator $\cK_u$ of a WSS $\gsp$ $u$ 
always extends to the identity operator on $\cF L^2(\mu_u)$, it may happen that $\cK_u$ fails to be continuous 
on $L^2(\rr d)$. In fact in this example $\cK_u$ is unbounded considered as an operator in $L^2(\rr d)$, 
equipped with domain $\cS(\rr d)$. 
\end{example}

\begin{rem}\label{rem:WSSunbounded}
In general it is not possible to consider $\cK_u$ as an unbounded operator in $L^2(\rr d)$. 
In fact suppose $\mu_u = \delta_0 \in \cM_+(\rr d)$. 
Then $\kappa_u = (2 \pi)^{-\frac{d}{2}} \cF^{-1} \mu_u = (2 \pi )^{-d}$
which implies that $\cK_u f(x) = \kappa_u * f(x) = (2 \pi )^{-d} \int_{\rr d} f(y) \dd y$.
Thus $\cK_u f \in L^2(\rr d)$ can happen only for functions $f$ with integral zero, and then $\cK_u f = 0$. 
There is no interesting function space (domain) for $f$ for which $\cK_u f \in L^2(\rr d) \setminus \{ 0\}$.  
For the stochastic process $u \in C (\rr d, L_0^2(\Omega) )$ that corresponds to $\mu_u = \delta_0$ we have 
$\bE | u(x) - u(0)|^2 = 0$ so $u(x) = u_0 \in L_0^2(\Omega)$ is constant for all $x \in \rr d$. 
\end{rem}

\begin{example}\label{ex:example2c}
We return to Example \ref{ex:example2b} and investigate the corresponding space 
$\cF L^2(\mu_u)$.
If $p = d = 1$ and $n \in \no$ then $\mu = \mu_{D^n u} =  |\xi|^{2 n} \dd \xi$, and hence
\begin{equation*}
\| f \|_{\cF L^2(\mu)}^2
= \int_{\ro} | \wh f(\xi) |^2 |\xi|^{2 n} \dd \xi.
\end{equation*}
This is the square norm of the homogeneous Hilbert Sobolev space $\dot H_n^2(\ro)$ \cite{Bergh1}. 
If $s \in \ro$ then by Remark \ref{rem:wssexistence} there exists a WSS tempered $\gsp$ $u$ on $\rr d$
such that $\mu_u = \eabs{\xi}^{2 s} \dd \xi$. 
In this case $\cF L^2(\mu) = H_s^2(\rr d)$ which denotes the usual Hilbert Sobolev space \cite{Bergh1}. 
\end{example}

\subsection{Filtering of WSS $\gsp$s}
\label{subsec:WSSfiltering}

Let $v \in \cL ( C_c^\infty(\rr d) , L_0^2(\Omega) )$ be WSS so that 
$C_c^\infty(\rr d) \ni \fy \mapsto \overline{( v,\fy) }$ extends to 
a unitary map $\cF L^2(\mu_v) \to L_v^2(\Omega)$
where $\mu_v \in \cM_+(\rr d)$ is the spectral Radon measure corresponding to $v$. 
Let $f \in \cF L^\infty (\mu_v)$. 
If $\fy \in \cF L^2 (\mu_v)$ then $\wh f \, \wh \fy \in L^2(\mu_v)$,
that is $\cF^{-1}( \wh f \, \wh \fy ) \in \cF L^2(\mu_v)$.
Defining a linear filter operator $F$ as the convolution 
\begin{equation}\label{eq:filterwss1}
F \fy = ( 2 \pi )^{-\frac{d}{2}} f * \fy =  \cF^{-1}( \wh f \, \wh \fy ), \quad \fy \in \cF L^2(\mu_v), 
\end{equation}
cf. \eqref{eq:convolutionFourier}, 
gives a translation invariant continuous operator on $\cF L^2(\mu_v)$. 
In fact
\begin{equation}\label{eq:filtertransinvbound}
( 2 \pi )^{-\frac{d}{2}} \| f * \fy \|_{\cF L^2 (\mu_v)} 
\leqs 
\| f \|_{\cF L^\infty (\mu_v)} \| \fy \|_{\cF L^2 (\mu_v)}.
\end{equation}

If $\fy, \psi \in \cF L^2(\mu_v)$ then we have, using $\wh f^* = \overline{ \wh f}$, 
\begin{align*}
(F \fy, \psi)_{\cF L^2(\mu_v)}  
& = \int_{\rr d} \wh f (\xi) \, \wh \fy (\xi) \, \overline{\wh \psi(\xi)} \, \dd \mu_v(\xi) \\
& = \int_{\rr d} \wh \fy (\xi) \, \overline{ \wh f^* (\xi) \, \wh \psi(\xi)} \, \dd \mu_v(\xi) \\
& = ( 2 \pi )^{-\frac{d}{2}} ( \fy, f^* * \psi)_{\cF L^2(\mu_v)} 
\end{align*}
which means that the adjoint of $F$ 
with respect to $(\cdot, \cdot)_{\cF L^2(\mu_v)}$ 
is 
\begin{equation}\label{eq:Fstarti}
F^* \fy = ( 2 \pi )^{-\frac{d}{2}} f^* * \fy. 
\end{equation}

We may hence define a filtered $\gsp$ $u_o \in \cL ( C_c^\infty(\rr d) , L_0^2(\Omega) )$,
cf. \eqref{eq:estimatorgsp}, as 
\begin{equation}\label{eq:filterwss2}
(u_o, \fy) = ( 2 \pi )^{-\frac{d}{2}} (v, f^* * \fy), \quad \fy \in \cF L^2 (\mu_v).
\end{equation}
The filtered GSP $u_o$ is WSS. Indeed since convolution is translation invariant as $g * T_x \fy = T_x( g * \fy)$, 
and $v$ is assumed to be WSS, 
we have for any $x \in \rr d$ and $\fy, \psi \in C_c^\infty(\rr d)$
\begin{align*}
( k_{u_o}, T_x \fy \otimes T_x \overline{\psi} )
& = \bE \left(  (u_o, T_x \fy) \overline{(u_o, T_x \psi)} \right) \\
& = ( 2 \pi )^{-d} \, \bE \left(  (v, f^* * T_x \fy) \overline{(v, f^* * T_x \psi)} \right) \\
& = ( 2 \pi )^{-d} \, \bE \left(  (v, f^* * \fy) \overline{(v, f^* * \psi)} \right) 
= ( k_{u_o}, \fy \otimes \overline{\psi} ). 
\end{align*}
%

\subsection{The optimal filter for WSS $\gsp$s}
\label{subsec:WSSoptfilter}

The next lemma will be needed in the proof of Theorem \ref{thm:optimalwss}. 

\begin{lem}\label{lem:measuremonotonic}
Let $\mu \in \cM_+(\rr d)$.
If $f \in L_{\rm loc}^1 (\mu)$ is real-valued
then 
\begin{equation*}
\int_A f(x) \, \dd \mu \geqs 0 \quad \forall A \in \cB(\rr d) \quad \Longrightarrow \quad f(x) \geqs 0 \ \, \mbox{for $\mu$-a.e.} \ \, x \in \rr d.  
\end{equation*}
\end{lem}

\begin{proof}
Set $N = \{x \in \rr d: \, f(x) \leqs 0 \} \in \cB(\rr d)$. Then 
\begin{equation*}
0 \leqs \int_N f(x) \, \dd \mu \leqs 0
\end{equation*}
that is $\int_N f(x) \, \dd \mu = 0$. 
Now \cite[Theorem~1.39~(a)]{Rudin2} implies $f (x) = 0$ for $\mu$-a.e. $x \in N$. 
\end{proof}

The following result is the main statement of this section. 
It reveals that the optimal filter for the uncorrelated additive noise problem for WSS $\gsp$s is 
a Radon--Nikodym derivative in the frequency domain. 
It formalizes the well established convention in engineering to think of the Wiener filter
in the frequency domain
as a fraction of spectral densities \cite{Fomin1,Vantrees1,Weinstein1,Wiener1}.

\begin{thm}\label{thm:optimalwss}
Suppose that $u,w \in \cL ( C_c^\infty(\rr d) , L_0^2(\Omega) )$ are two uncorrelated WSS generalized stochastic processes, 
let the corresponding spectral tempered non-negative Radon measures be denoted by $\mu_u, \mu_w \in \cM_+(\rr d)$ respectively, 
and let $v = u + w$. 

The unique convolution filter $f \in \cF L^\infty( \mu_u + \mu_w)$ 
that applied to $v$ 
as in \eqref{eq:filterwss2} 
which is mean square error optimal to recover $u$
is $f = \cF^{-1} \wh f \in \cS'(\rr d)$ where $\wh f$ is the Radon--Nikodym derivative given by
\begin{equation}\label{eq:RNconclusion}
\mu_u = \wh f \, (\mu_u + \mu_w). 
\end{equation}
We have $0 \leqs \wh f (\xi)\leqs 1$ for $(\mu_u+\mu_w)$-a.e. $\xi \in \rr d$, 
\begin{equation}\label{eq:supportfhat}
\esssupp \wh f = \supp \mu_u,
\end{equation}
and
the minimal error for $\fy \in \cF L^2(\mu_u + \mu_w)$ is 
\begin{equation}\label{eq:msegspwss}
J(\fy)
= \bE \left| (u - u_o, \fy) \right|^2
= \int_{\rr d} \wh f (\xi) \, |\wh \fy (\xi)|^2 \, \dd \mu_w 
= \int_{\rr d} \left( 1 - \wh f (\xi) \right) |\wh \fy (\xi)|^2 \, \dd \mu_u. 
\end{equation}
\end{thm}

\begin{proof}
We denote the auto-covariance operators for $u$ and $w$ by $\cK _u$ and $\cK _w$ respectively, 
and we set $\mu = \mu_u + \mu_w$ with domain $\cB(\rr d)$. 
The measures $\mu_u$, $\mu_w$ and $\mu$
satisfy \eqref{eq:temperedmeasure} for some common $s \geqs 0$. 
All three measures being non-negative entails
\begin{equation*}
\cF L^2(\mu) \subseteq \cF L^2(\mu_u) \bigcap \cF L^2(\mu_w). 
\end{equation*}

From \eqref{eq:gspwss2} it follows that for $\fy \in C_c^\infty(\rr d)$ we have
\begin{equation*}
\| (u, \fy ) \|_{L^2(\Omega)} \leqs \| \fy \|_{\cF L^2(\mu)}, \quad
\| (w, \fy ) \|_{L^2(\Omega)} \leqs \| \fy \|_{\cF L^2(\mu)}, 
\end{equation*}
which implies that $u$ and $w$ restricts to continuous linear maps 
$u: \cF L^2(\mu) \to L_u^2(\Omega)$ and 
$w: \cF L^2(\mu) \to L_w^2(\Omega)$, respectively. 
From \eqref{eq:bilinearbound} we obtain
\begin{align*}
& |( \cK _u \psi, \fy)| \leqs \| \psi \|_{ \cF L^2(\mu)} \, \| \fy \|_{ \cF L^2(\mu)}, \\
& |( \cK _w \psi, \fy)| \leqs \| \psi \|_{ \cF L^2(\mu)} \, \| \fy \|_{ \cF L^2(\mu)}, \quad \fy, \psi \in \cF L^2(\mu), 
\end{align*}
and by \eqref{eq:bilin2lin} and the preceding discussion we may regard
$\cK _u = \cT_{u,\mu}$ and $\cK _w = \cT_{w,\mu}$
as bounded operators on $\cF L^2(\mu)$ with operator norm upper bounded by one.
Defining $F$ by \eqref{eq:filterwss1} for $f \in \cF L^\infty( \mu)$
we know by \eqref{eq:filtertransinvbound}, \eqref{eq:Fstarti} 
that $\| F \|_{\cL (\cF L^2(\mu))} \leqs \| f \|_{\cF L^\infty( \mu)}$
and $\| F^* \|_{\cL (\cF L^2(\mu))} \leqs \| f \|_{\cF L^\infty( \mu)}$.

These considerations lead to the conclusion that we may consider the operator equation \eqref{eq:operatoreq2}, that is 
$\cK_u = F ( \cK_u +  \cK_w)$, for the optimal filter $F$ 
as an equation for bounded operators on $\cF L^2(\mu)$.
By Lemma \ref{lem:operatorequation} and \eqref{eq:Fstarti} this equation is equivalent to 
the equality
\begin{equation}\label{eq:measureequality1}
\begin{aligned}
\int_{\rr d} | \wh \fy (\xi)|^2 \, \dd \mu_u (\xi)
& = ( \cK_u \fy, \fy) 
= (F ( \cK_u + \cK_w) \fy, \fy ) 
= ( 2 \pi )^{-\frac{d}{2}} (( \cK_u + \cK_w) \fy, f^* * \fy) \\
& = ( 2 \pi )^{-\frac{d}{2}} \Big( ( \cK_u \fy,  f^* * \fy)_{\cF L^2(\mu_u)}  + ( \cK_w \fy, f^* * \fy)_{\cF L^2(\mu_w)} \Big) \\
& = ( 2 \pi )^{-\frac{d}{2}} \int_{\rr d} \wh \fy (\xi) \, \overline{ \wh {f^* * \fy }(\xi) }\, \dd \mu (\xi) \\
& = 
\int_{\rr d} |\wh \fy (\xi)|^2 \, \wh f (\xi) \, \dd \mu (\xi)
\end{aligned}
\end{equation}
for all $\fy \in \cS(\rr d)$. 
Due to the density of $\cF C_c^\infty(\rr d)$ in $\cS(\rr d)$, the equation \eqref{eq:measureequality1}
for all $\fy \in \cS(\rr d)$ is equivalent to the same equation for all 
$\fy \in \cF C_c^\infty(\rr d)$. 

Again the equation \eqref{eq:measureequality1} for all $\wh \fy \in C_c^\infty(\rr d)$ is equivalent to the same equation 
with $| \wh \fy (\xi)|^2$ replaced by any non-negative function $g \in C_c(\rr d)$. 
In fact let $\psi \in C_c^\infty (\rr d)$ fulfill $\supp \psi \subseteq \rB_1$, $\psi \geqs 0$, $\int \psi \dd x = 1$, set $\psi_n (x) = n^{d} \psi(x n)$ for $n \in \no \setminus 0$, 
let $\chi \in C_c^\infty(\rr d)$ satisfy $0 \leqs \chi \leqs 1$, $\chi |_{\supp g} \equiv 1$, 
and set 
\begin{equation*}
g_n = \chi \sqrt{ g * \psi_n + \frac1n} \in C_c^\infty(\rr d), \quad n \in \no \setminus 0. 
\end{equation*}
Then 
\begin{equation*}
\sup_{\xi \in \rr d} \left|  | g_n (\xi) |^2 - g(\xi) \right| \to 0, \quad n \to \infty, 
\end{equation*}
which implies that \eqref{eq:measureequality1} holds for all $\fy \in \cF C_c^\infty(\rr d)$
if and only if 
\begin{equation*}
\int_{\rr d} g(\xi) \, \dd \mu_u (\xi)
= \int_{\rr d} g(\xi) \, \wh f (\xi) \, \dd \mu (\xi)
\end{equation*}
holds for all non-negative functions $g \in C_c(\rr d)$. 
This means that the operator equation \eqref{eq:operatoreq2} reduces to 
the equation for tempered measures
\begin{equation}\label{eq:operatoreqwss1}
\mu_u = \wh f \, \mu = \wh f \left( \mu_u + \mu_w \right). 
\end{equation}

Since $\mu_u, \mu_w$ are non-negative measures, it follows that $\mu_u$ is absolutely continuous
with respect to $\mu$. 
By the Radon--Nikodym theorem \cite[Theorem~3.8]{Folland1} there exists a $\mu$-measurable function $\wh f \geqs 0$ $\mu$-a.e., with $\wh f \in L_{\rm loc}^1 (\mu)$,
uniquely determined $\mu$-a.e., such that \eqref{eq:operatoreqwss1} is satisfied in the sense of 
\begin{equation}\label{eq:measRNformula}
\mu_u (A) = \int_A \wh f (\xi) \, \dd \mu(\xi), \quad A \in \cB(\rr d).
\end{equation}
We have now shown \eqref{eq:RNconclusion}.

If $A \in \cB(\rr d)$ then
\begin{equation*}
\int_A \dd \mu (\xi) 
\geqs \int_A \dd \mu_u (\xi) 
= \mu_u (A) 
= \int_A \wh f (\xi) \, \dd \mu (\xi) 
\end{equation*}
that is
\begin{equation*}
\int_A (1 - \wh f (\xi)  )\dd \mu (\xi) \geqs 0, \quad A \in \cB(\rr d).  
\end{equation*}
An application of Lemma \ref{lem:measuremonotonic} yields $\wh f (\xi) \leqs 1$ $\mu$-a.e.
It follows that 
$\wh f \in L^\infty(\mu)$, that is $f \in \cF L^\infty( \mu)$. 
Thus the Radon--Nikodym derivative $\wh f$ is indeed the unique solution in $\cF L^\infty( \mu)$ to \eqref{eq:operatoreqwss1}. 

In order to show \eqref{eq:supportfhat} 
let $\Omega \subseteq \rr d \setminus \esssupp \wh f$ be open. Then $\wh f \big|_\Omega = 0$ $\mu$-a.e. \cite{Lieb1}
which implies $\mu_u(\Omega) = 0$ by \eqref{eq:measRNformula}, and thus $\Omega \subseteq \rr d \setminus \supp \mu_u$. 
This shows 
\begin{equation}\label{eq:suppfhat2a}
\esssupp \wh f \supseteq \supp \mu_u. 
\end{equation}
If instead $\Omega \subseteq \rr d \setminus \supp \mu_u$ is open then $\mu_u(\Omega) = 0$. 
Again from \eqref{eq:measRNformula} combined with \cite[Proposition~2.16]{Folland1} it follows that 
$\wh f \big|_\Omega = 0$ $\mu$-a.e. Hence $\Omega \subseteq \rr d \setminus \esssupp \wh f$ and we have shown 
\begin{equation}\label{eq:suppfhat2b}
\esssupp \wh f \subseteq \supp \mu_u. 
\end{equation}
Now \eqref{eq:supportfhat} is a consequence of \eqref{eq:suppfhat2a} and \eqref{eq:suppfhat2b}. 

Finally we obtain from \eqref{eq:minmse2} for $\fy \in \cF L^2(\mu)$, using 
\eqref{eq:operatoreq2}, 
\eqref{eq:gspwss1}, \eqref{eq:Fstarti} and \eqref{eq:operatoreqwss1}
\begin{align*}
J(\fy) 
& = ( (I-F)\cK_u  \fy, \fy )
= (\cK_u  \fy, ( I - F^*) \fy )
= ( \fy, ( I - F^* ) \fy)_{\cF L^2(\mu_u)} \\
& =  \int_{\rr d} \left( 1 - \wh f (\xi) \right) |\wh \fy (\xi)|^2 \, \dd \mu_u \\
& = \int_{\rr d} \wh f (\xi) |\wh \fy (\xi)|^2 \, \dd \mu_w
\end{align*}
which proves \eqref{eq:msegspwss}. 
\end{proof}

\begin{rem}\label{rem:translationinvariance}
Note that $J(T_x \fy) = J(\fy)$, that is $J$ is translation invariant. 
This feature is an intuitively natural consequence of the WSS assumption. 
The frequency support of $\fy \in \cF C_c(\rr d) \subseteq \cF L^2(\mu)$ affects the minimal error as follows:
\begin{align*}
& \supp \wh \fy \cap \supp \mu_w = \emptyset \quad \Longrightarrow \quad J(\fy) = 0, \\
& \supp \wh \fy \cap \supp \mu_u = \emptyset \quad \Longrightarrow \quad J(\fy) = 0.
\end{align*}
\end{rem}

The following consequence of Theorem \ref{thm:optimalwss} states a sufficient condition 
for perfect reconstruction (zero error) of the GSP $u$. 

\begin{cor}\label{cor:zeroerror}
Under the assumptions of Theorem \ref{thm:optimalwss}, and 
\begin{equation*}
\supp \mu_u \cap  \supp \mu_w = \emptyset
\end{equation*}
we have $\wh f = \chi_{\supp \mu_u}$ and $J(\fy) = 0$ for all $\fy \in \cF L^2(\mu)$. 
\end{cor}

\begin{proof}
By \eqref{eq:supportfhat} we have $\esssupp \wh f = \supp \mu_u$
and \eqref{eq:msegspwss} then yields $J(\fy) = 0$ for all $\fy \in \cF L^2(\mu)$.
The claim $\wh f = \chi_{\supp \mu_u}$ follows from \eqref{eq:RNconclusion}
and $\esssupp \wh f  \cap  \supp \mu_w = \emptyset$.
\end{proof}

\begin{example}\label{ex:example12c}
Suppose as in Remark \ref{rem:WSSabscontmeas} that the measure $\mu_u$ is absolutely continuous with respect to Lebesgue measure on $\rr d$, 
that is $\dd \mu_u = f_u \dd \xi$ for some $f_u \in L_{{\rm loc}}^1(\rr d)$. 
If the GSP $w$ is white noise with power $p > 0$ then by Theorem \ref{thm:optimalwss}
and Example \ref{ex:example1b} the optimal filter is 
\begin{equation*}
\wh f (\xi) = \frac{f_u(\xi)}{f_u(\xi) + p}, \quad \xi \in \rr d.
\end{equation*}
If the GSP $w$ is white noise with power $p > 0$ differentiated to order $\alpha \in \nn d$ as in 
Example \ref{ex:example2b} then 
\begin{equation*}
\wh f (\xi) = \frac{f_u(\xi)}{ f_u(\xi) + p \xi^{2 \alpha}}. 
\end{equation*}
So if $f_u$ has slower growth that $\xi^{2 \alpha}$
then the optimal filter will attenuate high frequencies
(``low-pass filter" \cite{Vantrees1}). 
\end{example}

\subsection{The optimal error for WSS stochastic processes}
\label{subsec:optimalwsssp}

When the WSS GSPs $u$ and $w$ in Theorem \ref{thm:optimalwss}
can be identified with stochastic processes
we get an error that does not depend on a test function, under certain restrictions. 
As we will see the error is in fact constant with respect to $x \in \rr d$.

Let $u \in C(\rr d, L_0^2(\Omega) )$ be a continuous WSS stochastic process with covariance function $k_u (x,y) = \bE \left( u(x) \overline{ u (y) } \right) = \kappa_u(x-y)$, $x,y \in \rr d$. 
By Bochner's theorem $\mu_u = (2 \pi)^{\frac{d}{2}} \cF \kappa_u$ is a non-negative bounded Radon measure. 
Thus 
\begin{equation}\label{eq:wssfourier}
\kappa_u(x) = (2 \pi)^{- d} \int_{\rr d} e^{i \la x, \xi \ra} \, \dd \mu_u(\xi), \quad x \in \rr d, 
\end{equation}
and $\kappa_u \in (C \cap L^\infty)(\rr d)$.

Let $f \in \cF C_c^\infty(\rr d)$ and define the convolution
\begin{equation}\label{eq:convolutionwss1}
u_o(x) = (2 \pi)^{- \frac{d}{2}} \int_{\rr d} f(x-y) u(y) \, \dd y, \quad x \in \rr d. 
\end{equation}
Then 
\begin{equation}\label{eq:variancefiltered}
\begin{aligned}
\| u_o (x) \|_{L^2(\Omega)}^2 
& = (2 \pi)^{- d} \iint_{\rr {2d}} f(x-y) \overline{f(x-z)} \, \kappa_u(y-z) \, \dd y \, \dd z \\
& = (2 \pi)^{- d} (f * \kappa_u, f) 
= (2 \pi)^{- d}  \int_{\rr d} | \wh f (\xi) |^2 \dd \mu_u 
\end{aligned}
\end{equation}
for any $x \in \rr d$. 
Thus the convolution \eqref{eq:convolutionwss1} extends to a well defined stochastic process $u_o \in L^\infty(\rr d, L_0^2(\Omega) )$
when $f \in \cF L^\infty (\mu_u)$, 
and $\| u_o (x) \|_{L^2(\Omega)}$ does not depend on $x \in \rr d$. 
The filtered stochastic process $u_o$ is thus well defined for $f \in \cF L^\infty (\mu_u)$. 
Note that $\cF L^\infty (\mu_u) \subseteq \cF L^2 (\mu_u)$ since the measure $\mu_u$ is bounded. 
We have
\begin{align*}
& \| u_o (x+y) - u_o(x) \|_{L^2(\Omega)}^2 \\
& = (2 \pi)^{- d} \iint_{\rr {2d}} f(z) \overline{f(w)} \, \left( u(x+y-z) - u(x-z), u(x+y-w) - u(x-w) \right)_{L^2(\Omega)} \, \dd z \, \dd w \\
& = (2 \pi)^{- d} \iint_{\rr {2d}} f(z) \overline{f(w)} \, \Big( 2 \kappa_u(w-z) - \kappa_u(w-z+y) - \kappa_u(w-z-y) \Big) \, \dd z \, \dd w \\
& = (2 \pi)^{- d} \int_{\rr {d}} \Big( 2 f * \kappa_u (w) - f * \kappa_u (w+y) - f * \kappa_u  (w-y) \Big) \overline{f(w)} \, \dd w \\
& = (2 \pi)^{- d} \left( 2 \wh{f * \kappa_u} - M_{y} \wh{ f * \kappa_u} - M_{-y} \wh{f * \kappa_u}, \wh f \right)_{L^2} \\
& = (2 \pi)^{- d} \int_{\rr {d}} | \wh f (\xi) |^2 \, \left(2 - e^{i \la y, \xi \ra}  - e^{-i \la y, \xi \ra} \right) \, \dd \mu_u  \\
& = 2 (2 \pi)^{- d} \int_{\rr {d}} | \wh f (\xi) |^2 \, \left( 1 - \cos( \la y, \xi \ra ) \right) \, \dd \mu_u 
\end{align*}
so it follows from dominated convergence that $u_o \in C(\rr d, L_0^2(\Omega) )$ if $f \in \cF L^\infty (\mu_u)$. 

Let $v = u + w$ where $u,w$ are uncorrelated continuous stochastic processes. 
We pick $f \in \cF L^\infty (\mu_u + \mu_w)$ in \eqref{eq:convolutionwss1} in order to minimize the mean square error $\bE | u_o(x) - u(x)|^2$
where $u_o$ is defined by \eqref{eq:convolutionwss1} with $u$ replaced by $v$.
This is a particular case of Theorem \ref{thm:optimalwss}, and hence $\wh f \in L^\infty (\mu_u + \mu_w)$ is to be chosen as the Radon--Nikodym derivative, that by definition satisfies
\begin{equation}\label{eq:filterwssRN}
\mu_u = \wh f \, ( \mu_u + \mu_w). 
\end{equation}

Using \eqref{eq:wssfourier}, \eqref{eq:convolutionwss1}, \eqref{eq:variancefiltered} with $\mu_u$ replaced by $\mu_v = \mu_u + \mu_w$, 
\eqref{eq:filterwssRN}, $\kappa_u^* = \kappa_u$ and $\wh f \geqs 0$
we obtain the minimal mean square error as
\begin{align*}
\| u_o(x) - u(x) \|_{L^2(\Omega)}^2
& = \| u_o(x) \|_{L^2(\Omega)}^2 + \| u(x) \|_{L^2(\Omega)}^2 - 2 \, \re ( u_o(x), u(x) )_{L^2(\Omega)} \\
& = (2 \pi)^{- d}  \int_{\rr d} | \wh f (\xi) |^2 (\dd \mu_u + \dd \mu_w) + \kappa_u(0) - 2 (2 \pi)^{- \frac{d}{2}} \re (f, \kappa_u)_{L^2} \\
& = (2 \pi)^{- d}  \int_{\rr d} \left( \wh f (\xi) + 1 - 2 \wh f (\xi) \right) \dd \mu_u  \\
& = (2 \pi)^{- d}  \int_{\rr d} \left( 1 - \wh f (\xi) \right) \dd \mu_u 
= (2 \pi)^{- d}  \int_{\rr d} \wh f (\xi) \, \dd \mu_w 
\end{align*}
independently of $x \in \rr d$, as claimed.

\begin{rem}\label{rem:wssminmse}
Writing 
\begin{equation*}
J = 2 (2 \pi)^{d} \| u_o(x) - u(x) \|_{L^2(\Omega)}^2
= \int_{\rr d} \left( 1 - \wh f (\xi) \right) \dd \mu_u 
+ \int_{\rr d} \wh f (\xi) \, \dd \mu_w 
= J_1 + J_2
\end{equation*}
we split the total error as $J = J_1 + J_2$ where the bargain of the optimal filter $\wh f$ becomes visible.  
In fact, recalling $0 \leqs \wh f(\xi) \leqs 1$, $\wh f$ should be close to 1 at the frequencies where the signal in $\mu_u$ 
is strong in order to minimize $J_1$. 
On the other hand $\wh f$ should be close to zero at the frequencies where the noise in $\mu_w$ is strong in order to attenuate $J_2$. 
The compromise may give a small error if $\supp \mu_u \cap \supp \mu_w$ is small.
\end{rem}

\begin{rem}\label{rem:msegspwss}
For a WSS generalized stochastic process the minimal error depends on the test function $\fy \in \cF L^2(\mu)$ according to 
\eqref{eq:msegspwss}. 
Also in this case we may split the minimal error as
\begin{equation*}
2 J(\fy)
= \int_{\rr d} \left( 1 - \wh f (\xi) \right) |\wh \fy (\xi)|^2  \, \dd \mu_u
+ \int_{\rr d} \wh f (\xi) \, |\wh \fy (\xi)|^2 \, \dd \mu_w.
\end{equation*}
This admits a similar interpretation as for a stochastic process (Remark \ref{rem:wssminmse}), with the modification of the test function frequency weight $|\wh \fy (\xi)|^2$
in the integrals. 
\end{rem}

\section{Optimal filtering of non-stationary generalized stochastic processes}
\label{sec:optfiltnonstat}

Let $u,w  \in \cL ( C_c^\infty(\rr d), L_0^2(\Omega) )$ be uncorrelated generalized stochastic processes. 
In this section we do not assume that $u$ and $w$ are WSS.
We return to the model \eqref{eq:signalplusnoise} and the corresponding non-stationary operator equation \eqref{eq:operatoreq2}
\begin{equation}\label{eq:operatoreq3}
\cK_u = F (\cK_u + \cK_w)
\end{equation}
in the space of continuous linear operators $C_c^\infty(\rr d) \to \cD'(\rr d)$. 
Recall that $\cK_u \geqs 0$, $\cK_w \geqs 0$, and that $F$ is assumed to be a continuous linear operator on $C_c^\infty (\rr d)$
which extends uniquely to a continuous linear operator on $\cD'(\rr d)$.

The equation \eqref{eq:operatoreq3} when $u$ or $w$ or both are not WSS 
is more cumbersome to handle than the case when both are WSS. 
In the latter case we could in Section \ref{sec:optimalwss} for a given WSS $u  \in \cL ( \cS(\rr d), L_0^2(\Omega) )$ choose a Hilbert subspace $\cF L^2(\mu_u) \subseteq \cS'(\rr d)$ where $\mu_u \in \cM_+ (\rr d)$ is the spectral measure corresponding to $u$,
such that the covariance operator $\cK_u$ extends to the identity operator on $\cF L^2(\mu_u)$. 
Note that the space $\cF L^2(\mu_u)$ is adapted to the specific $u  \in \cL ( \cS(\rr d), L_0^2(\Omega) )$. 
In the non-stationary case, without further restrictions, there is no natural space on which a covariance operator is continuous. 

The covariance operators $\cK$ for WSS $\gsp$s are translation invariant as $\cK_u T_x = T_x \cK_u$ for all $x \in \rr d$, that is they are convolution operators. Moreover the space $\cF L^2(\mu_u)$ is also translation invariant (more precisely translation isometric).

\subsection{Commuting covariance operators}\label{subsec:commuting}

In this subsection we start from the observation that the operators $\cK_u$ and $\cK_w$ commute
when $u$ and $w$ are WSS $\gsp$s, since $\cK_u$ and $\cK_w$ are convolution operators then. 

We let $u,w  \in \cL ( \cS(\rr d), L_0^2(\Omega) )$, assume they are uncorrelated as before, allow them to be non-stationary, but we keep the restriction that $\cK_u$ and $\cK_w$ commute. 
We also assume that $\cK_u, \cK_w \in \cL( \cH )$, $\cK_u \geqs 0$ and $\cK_w \geqs 0$ 
where $\cH$ is a Hilbert space of distributions defined on $\rr d$. 
This framework is realized e.g. with modulation spaces if $\cH = M_{t,s}^2(\rr d)$
and the covariance operators are Weyl pseudodifferential operators with 
symbols in $M_{0,0, |s|, |t|}^{\infty,1}(\rr {2d})$ for any $t,s \in \ro$
as we discussed in 
Section \ref{subsec:modspace}, cf. \eqref{eq:contpsdomodsp2}. 

Then we may use the powerful methods that are based on the spectral theorem. 
The spectral theorem (cf. \cite[Theorem~IX  \S 2.2.2]{Conway1} and \cite[Theorem~9.14]{vanNeerven1}) concerns a normal operator $N \in \cL( \cH )$
which by definition satisfies $N^* N = N N^*$. 
For such an operator there exists a measure $\Pi$ 
with domain $\cB( \sigma (N) )$ where $\sigma(N) \Subset \co$ denotes the spectrum of $N$, 
and values in the projections in $\cL( \cH )$,
such that $N$ can be spectrally decomposed as 
\begin{equation*}
N = \int_{\sigma(N)} z \dd \Pi (z). 
\end{equation*}
By \cite[Theorem~IX  \S 2.2.3]{Conway1} or \cite[Theorem~9.8]{vanNeerven1} there exists a bounded functional calculus 
defined by
\begin{equation*}
f(N) = \int_{\sigma(N)} f(z) \dd \Pi (z) \in \cL( \cH )
\end{equation*}
for $f \in L^\infty( \sigma(N), \Pi )$ (cf. \cite[Proposition~9.11]{vanNeerven1}) which is a $^*$-homomorphism and an isomorphism. 
If $N$ is selfadjoint then $\supp \Pi \subseteq \ro$ and if $N \geqs 0$ then $\supp \Pi \subseteq \ro_+$ (cf. \cite[Corollary~9.18]{vanNeerven1}). 

We use von Neumann's result \cite[Theorem~9.24]{vanNeerven1}
which depends on the spectral theorem for normal operators on a Hilbert space. 
This theorem concerns in fact the more general assumption of two commuting self-adjoint operators. 
Adapted to our setup of commuting non-negative operators, the theorem and 
its proof says the following. 
There exists a measure $\Pi$ with domain $\cB(\co)$, compactly supported in the first quadrant, i.e. 
$\supp \Pi \Subset \{x+iy \in \co:  x \geqs 0, \ y \geqs 0 \} := \co_+$, 
which is projection-valued in $\cL( \cH )$, 
and there exists a normal operator $N \in \cL( \cH )$ with spectral decomposition 
\begin{equation*}
N = \int_{\co_+} z \dd \Pi (z). 
\end{equation*}
Moreover there exist two functions $f_u, f_w \in C_c( \co_+ )$
such that $f_u, f_w \geqs 0$, and $\cK_u$ and $\cK_w$ can
be simultaneously spectrally decomposed with respect to $\Pi$ as  
\begin{align}
\cK_u & = \int_{\co_+} f_u(z) \dd \Pi (z), \label{eq:Ku} \\
\cK_w & = \int_{\co_+} f_w(z) \dd \Pi (z). \label{eq:Kw}
\end{align}

If we set
\begin{equation}\label{eq:optspectralfunc}
f(z) = \frac{f_u(z)}{f_u(z) + f_w(z)} \, \chi_{\supp f_u}
\end{equation}
then $f: \co_+ \to \ro_+$ is a well defined bounded measureable function.  
The functional calculus of the operator $N$ hence yields that 
\begin{equation}\label{eq:optimalcommute}
F = \int_{\co_+} f(z) \dd \Pi (z) \in \cL( \cH )
\end{equation}
is an operator that solves \eqref{eq:operatoreq3}.
The solution $F$ is unique in the bounded functional calculus of $N$. 
The fact that $f \geqs 0$ implies by \cite[Theorem~9.8]{vanNeerven1}
that $F$ is non-negative. 

We have shown the following result. 

\begin{thm}\label{thm:commutesol}
Let $u,w  \in \cL ( \cS(\rr d), L_0^2(\Omega) )$ be uncorrelated. 
Suppose that the covariance operators satisfy $\cK_u, \cK_w \in \cL( \cH )$ 
where $\cH$ is a Hilbert space of distributions defined on $\rr d$, 
$\cK_u \geqs 0$, $\cK_w \geqs 0$ and $\cK_u \cK_w = \cK_w \cK_u$.

Then there exists a Borel measure $\Pi$, compactly supported on $\co_+$, with values in the projections in $\cL( \cH )$, 
such that $\cK_u$ and $\cK_w$ can be expressed as \eqref{eq:Ku} and \eqref{eq:Kw} respectively, 
with non-negative $f_u, f_w \in C_c( \co_+ )$, 
and the solution to \eqref{eq:operatoreq3} is given by \eqref{eq:optspectralfunc} and \eqref{eq:optimalcommute}. 
The solution satisfies $F \geqs 0$ and is unique in the bounded functional calculus of $\Pi$. 
\end{thm}

\subsection{Pseudodifferential operators}\label{subsec:pseudodiff}

In this subsection we relax both the conditions that $u,w$ are WSS and that $\cK_u$ and $\cK_w$
commute. 
The covariance operators are allowed to be pseudodifferential operators rather than convolution operators as in Section \ref{sec:optimalwss}.
We use the functional framework of Section \ref{subsec:modspace}, that is operators with symbols in modulation spaces, 
acting on modulation spaces. 
In an effort to extend the analysis in Section \ref{sec:optimalwss} we try to choose a framework that contains the WSS case as far as possible.

First we deduce an inclusion of modulation spaces into the spaces on which WSS covariance operators act in Section \ref{sec:optimalwss}. 
This gives a relation between spaces used in Section \ref{sec:optimalwss} and spaces used in this section. 

\begin{prop}\label{prop:modspinclusions}
Let $1 \leqs p,q < \infty$ and let $\mu \in \cM_+(\rr d)$ 
with $s \geqs 0$ chosen such that \eqref{eq:temperedmeasure} is satisfied. 
If $\ep > 0$ and 
\begin{equation*}
\omega (x,\xi) = \eabs{x}^{\frac{1}{p'}(d + \ep)} \eabs{\xi}^{ \frac{s}{2} + \frac{1}{q'}(d + \ep)}
\end{equation*}
then 
\begin{equation}\label{eq:modspembeddings}
M_\omega^{p,q}(\rr d) \subseteq \cF L^2 (\mu). 
\end{equation}
\end{prop}

\begin{proof}
Let $f \in \cS(\rr d)$
and let $\fy \in \cS(\rr d)$ satisfy $\| \fy \|_{L^2} = 1$. 
We use \eqref{eq:STFTinverse} which gives
\begin{equation*}
\wh f = (2\pi )^{-\frac d2} \iint_{\rr {2d}} V_\fy f (x,\xi) T_\xi M_{-x} \wh \fy \, \dd x \, \dd \xi.
\end{equation*}

Using \eqref{eq:Peetre} and H\"older's inequality we obtain for all $\eta \in \rr d$
\begin{align*}
\eabs{\eta}^{\frac{s}{2}} |\wh f (\eta) | 
& \lesssim \iint_{\rr {2d}} |V_\fy f (x,\xi)| \, |\wh \fy (\eta - \xi) | \eabs{\eta}^{\frac{s}{2}} \, \dd x \, \dd \xi
\lesssim \iint_{\rr {2d}} |V_\fy f (x,\xi)| \eabs{\xi}^{\frac{s}{2}} \, \dd x \, \dd \xi \\
& = \iint_{\rr {2d}} |V_\fy f (x,\xi)| \, \omega(x,\xi) \, \eabs{x}^{- \frac{1}{p'}(d + \ep)} \eabs{\xi}^{- \frac{1}{q'}(d + \ep)} \dd x \, \dd \xi \\
& \lesssim \int_{\rr d} \| V_\fy f (\cdot,\xi) \, \omega(\cdot,\xi) \|_{L^p(\rr d)} \eabs{\xi}^{- \frac{1}{q'}(d + \ep)}  \dd \xi \\
& \lesssim \| (V_\fy f ) \omega \|_{L^{p,q}(\rr {2d})} 
= \| f \|_{M_\omega^{p,q}(\rr d)}. 
\end{align*}
From \eqref{eq:temperedmeasure} it hence follows that
\begin{equation}\label{eq:measFL2mod}
\| f \|_{\cF L^2 (\mu)}
= \left( \int_{\rr d} | \wh f(\eta) |^2 \dd \mu(\eta) \right)^{\frac12}
\lesssim \sup_{\eta \in \rr d} \eabs{\eta}^{\frac{s}{2}} | \wh f(\eta) |
\lesssim \| f \|_{M_\omega^{p,q}(\rr d)}.
\end{equation}
The inclusion \eqref{eq:modspembeddings} now follows from the density of $\cS(\rr d)$ in $M_\omega^{p,q}(\rr d)$ \cite[Proposition~11.3.4]{Grochenig1}.
\end{proof}

\begin{rem}\label{rem:translationinvariance}
The space $\cF L^2 (\mu)$ is translation invariant, and in fact also translation isometric. 
The modulation spaces $M_\omega^{p,q}$ are translation invariant but in general not translation isometric
\cite{Grochenig1}.
\end{rem}

\begin{rem}\label{rem:oppositeinclusion}
From \eqref{eq:STFTtempered} it follows
\begin{equation}\label{eq:tempdismodsp}
\cS'(\rr d) = \bigcup_{t \in \ro} M_{\omega_t}^{p,q}(\rr d)
\end{equation}
for any $p,q \in [1,\infty]$ with $\omega_t(X) = \eabs{X}^t$ and $X \in \rr {2d}$. 
From Section \ref{subsec:measures} we know that the inclusions
\begin{equation}\label{eq:inclusionsSSp}
\cS(\rr d) \subseteq \cF L^2 (\mu) \subseteq \cS'(\rr d)
\end{equation}
are continuous and dense but the left inclusion is in general not injective if $\mu \in \cM_+(\rr d)$. 
Combining \eqref{eq:tempdismodsp} and \eqref{eq:inclusionsSSp} we obtain the injective inclusion
\begin{equation*}
\cF L^2 (\mu) \subseteq \bigcup_{t \in \ro} M_{\omega_t}^{p,q}(\rr d)
\end{equation*}
which may be seen as an upper inclusion opposite to \eqref{eq:modspembeddings}.

However we cannot get continuous inclusions in individual modulation spaces of the form 
$\cF L^2 (\mu) \subseteq M_{\omega}^{p,q}(\rr d)$ in general, that is an estimate of the form
\begin{equation*}
\| f \|_{M_{\omega}^{p,q}}  \lesssim \| f \|_{\cF L^2 (\mu)}.
\end{equation*}
Indeed suppose $\supp \mu \Subset \rr d$, take $f \in \cS(\rr d) \setminus \{ 0 \}$ 
such that $\wh f = 0$ in $L^2 (\mu)$. Then $f \neq 0$ in $M_{\omega}^{p,q}(\rr d)$
for any $p,q \in [1,\infty]$ and any $\omega \in \mascP(\rr d)$.
\end{rem}

The covariance operator $\cK_u$ of a WSS $\gsp$ $u$ can be seen as a convolution operator as in \eqref{eq:WSSconvop}
and \eqref{eq:WSSconvopcont}, where $\kappa_u = (2 \pi)^{- \frac{d}{2} } \cF^{-1} \mu_u \in \cS'(\rr d)$
and $\mu_u \in \cM_+(\rr d)$.
From \eqref{eq:weylquantization} we may identify the Weyl symbol of $\cK_u = a_u^w(x,D)$ as $a_u = 1 \otimes \mu_u \in \cM_+(\rr {2d})$. 
Since we want to include the WSS case we need to work with pseudodifferential operators with symbols that are not smooth, as opposed to the generic frameworks for pseudodifferential operators when they are applied to PDEs \cite{Folland1,Hormander1,Shubin1}. In fact we need spaces of non-regular symbols, including elements of the form $1 \otimes \mu \in \cM_+(\rr {2d})$ with $\mu \in \cM_+(\rr d)$.
The next result shows that $1 \otimes \mu$ belongs to a modulation space when $\mu \in \cM_+(\rr d)$. 

\begin{prop}\label{prop:posmeasSTFT}
Suppose $\mu \in \cM_+(\rr d)$ and let $s \geqs 0$ be chosen so that 
\eqref{eq:temperedmeasure} is fulfilled. Then for any $n \in \no$ and any $\ep > 0$ we have $1 \otimes \mu \in M_{\omega}^{\infty,1}(\rr {2d})$
with 
\begin{equation}\label{eq:weightWSS1}
\omega(x_1, x_2, \xi_1, \xi_2) = \eabs{x_2}^{-s} \eabs{\xi_1}^n \eabs{\xi_2}^{-d - \ep}, \quad x_1, x_2, \xi_1, \xi_2 \in \rr d.
\end{equation}
\end{prop}

\begin{proof}
Let $\fy \in \cS(\rr d) \setminus 0$. We have 
\begin{equation*}
V_{\fy \otimes \fy}( 1 \otimes \mu)(x_1, x_2, \xi_1, \xi_2) = V_\fy 1(x_1, \xi_1) V_\fy \mu (x_2, \xi_2)
\end{equation*}
and $V_\fy 1(x_1, \xi_1) = e^{- i \la x_1, \xi_1 \ra} \overline{\wh \fy(-\xi_1)}$.
Using \eqref{eq:Peetre} and \eqref{eq:temperedmeasure} this gives for any $n \in \no$ and any $\ep > 0$
\begin{align*}
\left| V_{\fy \otimes \fy}( 1 \otimes \mu)(x_1, x_2, \xi_1, \xi_2) \right|
& = (2 \pi)^{-\frac{d}{2}} | \wh \fy(-\xi_1) | \left| \int_{\rr d} e^{- i \la \xi_2, y_2 \ra} \overline{\fy(y_2 - x_2)} \dd \mu(y_2) \right| \\
& \lesssim \eabs{\xi_1}^{-n - d - \ep} \int_{\rr d} \eabs{y_2 - x_2}^{-s} \dd \mu(y_2) \\
& \lesssim \eabs{\xi_1}^{-n - d - \ep} \eabs{x_2}^{s}.
\end{align*}
It follows that $\left| \big( V_{\fy \otimes \fy}( 1 \otimes \mu) \omega \big) (x_1, x_2, \xi_1, \xi_2 ) \right| \lesssim  \eabs{\xi_1}^{- d - \ep} \eabs{\xi_2}^{- d - \ep}$ which implies
\begin{equation*}
\| 1 \otimes \mu \|_{M_{\omega}^{\infty,1}(\rr {2d})} 
= \int_{\rr {2d}}\sup_{X \in \rr {2d} }  \Big( \left| \big( V_{\fy \otimes \fy}( 1 \otimes \mu) \omega \big) (X,\Xi) \right| \Big) \dd \Xi
< \infty.
\end{equation*}
\end{proof}

With the notation of Section \ref{subsec:modspace}, Proposition \ref{prop:posmeasSTFT} says that 
$1 \otimes \mu \in M_{0, - s, n, - d - \ep}^{\infty,1}(\rr {2d})$ for any $n \in \no$, any $\ep > 0$, and some $s \geqs 0$. 
From \eqref{eq:weightpowers} and the preceding discussion we get the following consequence. 

\begin{cor}\label{cor:contmodspWSS}
Suppose $\mu \in \cM_+(\rr d)$, let $s \geqs 0$ be chosen so that 
\eqref{eq:temperedmeasure} is fulfilled, set $a = 1 \otimes \mu$,
let $p,q \in [1,\infty]$, and let $t_j,s_j \in \ro$ for $j = 1,2$
satisfy for some $\ep > 0$
\begin{equation}\label{eq:weightcond}
t_1 \geqs d + \ep, \quad t_2 \leqs - d - \ep, \quad \mbox{and} \quad s_2 \leqs s_1 - s.
\end{equation}
Then 
$a^w(x,D): M_{t_1, s_1}^{p,q}(\rr d) \to M_{t_2, s_2}^{p,q}(\rr d)$
is continuous. 
\end{cor}

\begin{rem}\label{rem:targetspaces}
Since $t_2 < t_1$ and $s_2 \leqs s_1$ in the assumptions \eqref{eq:weightcond}
we have $M_{t_1, s_1}^{p,q}(\rr d) \subsetneq M_{t_2, s_2}^{p,q}(\rr d)$. 
Corollary \ref{cor:contmodspWSS} does therefore not give a common space for domain and range of the operator. 
\end{rem}

Proposition \ref{prop:posmeasSTFT} gives evidence that $M_\omega^{\infty,1}(\rr {2d})$ 
for certain weights $\omega \in \mascP(\rr {4d})$
may be an interesting class for covariance operators of non-stationary $\gsp$s that 
contains the covariance operators for WSS $\gsp$s. 
Corollary \ref{cor:contmodspWSS} says that the corresponding WSS covariance operators act continuously between certain modulation spaces, and  Proposition \ref{prop:modspinclusions} relates modulation spaces to the spaces $\cF L^2(\mu)$ used in Section \ref{sec:optimalwss}. 

Nevertheless we need to introduce some restrictions in order to proceed. 
In fact we want to use the spectral invariance results for modulation spaces discussed in Section \ref{subsec:modspace}. 
The space $M_{\omega}^{\infty,1}(\rr {2d})$ with the weight \eqref{eq:weightWSS1} 
is not quite good enough as a symbol class. 
We need instead $M_{1 \otimes \omega}^{\infty,1}(\rr {2d})$ with $\omega(X) = \eabs{X}^r$ for $r \geqs 0$. 

For the Weyl symbols of covariance operators of WSS $\gsp$s this leads to the following criterion. 

\begin{lem}\label{lem:WSSsymbolmodsp}
Let $f \in \cS'(\rr d)$, $r \geqs 0$, $\omega_1(\xi) = \eabs{\xi}^r$ for $\xi \in \rr d$, and 
$\omega_2(X) = \eabs{X}^r$ for $X \in \rr {2d}$.
Then $1 \otimes f \in M_{1 \otimes \omega_2}^{\infty,1}(\rr {2d})$
if and only if $f \in M_{1 \otimes \omega_1}^{\infty,1}(\rr d)$. 
\end{lem}

\begin{proof}
Let $\fy \in \cS(\rr d) \setminus 0$. 
As in the proof of Proposition \ref{prop:posmeasSTFT} we have
\begin{equation*}
|V_{\fy \otimes \fy}( 1 \otimes f)(x_1, x_2, \xi_1, \xi_2) | 
= |\wh \fy(-\xi_1)| \, |V_\fy f (x_2, \xi_2) |. 
\end{equation*}
This gives, using \eqref{eq:Peetre}, 
\begin{align*}
|\wh \fy(-\xi_1)| \, |V_\fy f (x_2, \xi_2) | \eabs{\xi_2}^r
& \leqs 
|V_{\fy \otimes \fy}( 1 \otimes f)(x_1, x_2, \xi_1, \xi_2) | \eabs{(\xi_1,\xi_2)}^r \\
& \lesssim |\wh \fy(-\xi_1)| \eabs{\xi_1}^r |V_\fy f (x_2, \xi_2) | \eabs{\xi_2}^r
\end{align*}
and it follows that 
\begin{align*}
\| \wh \fy \|_{L^1} \| f \|_{M_{1 \otimes \omega_1}^{\infty,1}(\rr d)}
& = \int_{\rr {2d}} |\wh \fy(-\xi_1)| \Big( \sup_{x_2 \in \rr d} |V_\fy f (x_2, \xi_2) | \Big) \eabs{\xi_2}^r \dd \xi_1 \dd \xi_2
\leqs
\| 1 \otimes f \|_{M_{1 \otimes \omega_2}^{\infty,1}(\rr {2d})} \\
& \lesssim \int_{\rr d} |\wh \fy(-\xi_1)| \eabs{\xi_1}^r \dd \xi_1
\int_{\rr d} \Big( \sup_{x_2 \in \rr d} |V_\fy f (x_2, \xi_2) | \Big) \eabs{\xi_2}^r  \dd \xi_2 \\
& \lesssim \| f \|_{M_{1 \otimes \omega_1}^{\infty,1}(\rr d)}. 
\end{align*}
\end{proof}

The requirement $\mu = f \in M_{1 \otimes \eabs{\cdot}^r}^{\infty,1}(\rr d)$ for some $r \geqs 0$ on the spectral measure $\mu \in \cM_+(\rr d)$ for a WSS $\gsp$  
is hence equivalent to its covariance operator having a Weyl symbol in 
$M_{1 \otimes \omega}^{\infty,1}(\rr {2d})$ with $\omega(X) = \eabs{X}^r$.

The condition $\mu \in M_{1 \otimes \eabs{\cdot}^r}^{\infty,1}(\rr d)$ is satisfied for 
white noise for any $r \geqs 0$. Indeed then we have $\mu = p \, \dd \xi$ for a constant $p > 0$, cf. Example \ref{ex:example1b}, 
and it follows from the proof of Lemma \ref{lem:WSSsymbolmodsp} that $1 \in M_{1 \otimes \eabs{\cdot}^r}^{\infty,1}(\rr d)$
for any $r \geqs 0$. 
But if $u \in \cL ( \cS(\rr d), L_0^2(\Omega) )$ is white noise of power $p = 1$ then $D^\alpha u$ has spectral measure
$\mu_{D^\alpha u} = \xi^{2 \alpha} \dd \xi$ if $\alpha \in \nn d$, cf. Example \ref{ex:example2b}.
The Weyl symbol of the covariance operator $\cK_{D^\alpha u}$ is therefore $1 \otimes \xi^{2 \alpha}$. 
This Weyl symbol does not belong to 
$M_{1 \otimes \omega}^{\infty,1}(\rr {2d})$ with
$\omega(X) = \eabs{X}^r$ for $X \in \rr {2d}$,
for any $r \geqs 0$ if $\alpha \neq 0$. 
In fact $M_{1 \otimes \omega}^{\infty,1}(\rr {2d}) \subseteq M^{\infty,1}(\rr {2d}) \subseteq (C \cap L^\infty) (\rr {2d})$
\cite{Sjostrand1}. 
Thus white noise is included but its derivatives are excluded 
when we work with the symbol classes 
$M_{1 \otimes \eabs{\cdot}^r }^{\infty,1}(\rr {2d})$ for some $r \geqs 0$. 
Furthermore $\mu \in M_{1 \otimes \eabs{\cdot}^r}^{\infty,1}(\rr d) \subseteq (C \cap L^\infty) (\rr d)$ implies that $\mu$ is absolutely continuous with respect to Lebesgue measure. 

\begin{rem}\label{rem:GSPmodsp}
Let $u \in \cL ( \cS(\rr d), L_0^2(\Omega) )$ and assume that its covariance operator satisfies 
$\cK_u = a_u^w(x,D)$ with $a_u \in M_{1 \otimes \omega}^{\infty,1}(\rr {2d})$
with $\omega(X) = \eabs{X}^r$ for some $r\geqs 0$. 
By \eqref{eq:contpsdomodsp2} and the surrounding discussion it follows that $\cK_u \in \cL( M^1 (\rr d) )$. 
Using $(M^1)' = M^\infty$
and \eqref{eq:modspnested}
this gives for $\psi, \fy \in M^1 (\rr d)$
\begin{equation*}
|( \cK_u \psi, \fy )| \leqs \| \cK_u \psi \|_{M^1(\rr d)} \| \fy \|_{M^\infty(\rr d)}
\lesssim \| \psi \|_{M^1(\rr d)} \| \fy \|_{M^1(\rr d)} 
\end{equation*}
and it follows that 
$u \in \cL ( M^1(\rr d), L_0^2(\Omega) )$, cf. Remark \ref{rem:OpEq}. 
The assumption $a_u \in M_{1 \otimes \omega}^{\infty,1}(\rr {2d})$ hence implies that 
the corresponding class of $\gsp$s is a subspace of the class studied in 
\cite{Feichtinger2,Hormann1,Keville1,Wahlberg1} with test function space Feichtinger's algebra. 
\end{rem}

As preparation for the next result on non-stationary filters we will discuss operators $\cK \in \cL( L^2(\rr d) )$ that satisfy an assumption of the form 
\begin{equation}\label{eq:lowerboundL2}
(\cK f, f) \geqs \ep \| f \|_{L^2(\rr d)}^2, \quad f \in L^2(\rr d), 
\end{equation}
for some $\ep > 0$. This assumption is a strengthening of $\cK \geqs 0$ on $L^2(\rr d)$. 

\begin{lem}\label{lem:lowerbound}
If $\cK \in \cL( L^2(\rr d) )$ satisfies $\cK \geqs 0$ then $\cK$ is invertible on $L^2(\rr d)$ 
if and only if \eqref{eq:lowerboundL2} is satisfied for some $\ep > 0$.
\end{lem}

\begin{proof}
Suppose \eqref{eq:lowerboundL2} is satisfied for some $\ep > 0$. Since $(\cK f, f) = \| \cK^{\frac12} f \|_{L^2}^2$
it follows that $\ker \cK^{\frac12} = \{ 0 \}$, $\ran \cK^{\frac12}$ is closed in $L^2$, and hence $\ran \cK^{\frac12} = \left( \ker \cK^{\frac12} \right)^\perp = L^2(\rr d)$. 
By the open mapping theorem $\cK^{\frac12}$ is invertible on $L^2$ and hence $\cK^{-1} = \left( \cK^{-\frac12} \right)^2 \in \cL(L^2)$. 

Suppose on the other hand that \eqref{eq:lowerboundL2} is not satisfied for any $\ep > 0$. 
Then there exists a sequence $\{ f_n \}_{n \geqs 1} \subseteq L^2(\rr d)$ such that $\| f_n \|_{L^2} = 1$ for all $n \geqs 1$ and 
$\lim_{n \to \infty} (\cK f_n, f_n) = 0$. 
It follows that 
\begin{equation}\label{eq:limitzero}
\lim_{n \to \infty} \| \cK^{\frac12} f_n \|_{L^2} = 0. 
\end{equation}
If we assume that $\cK$ is invertible on $L^2$ then $\cK^{-1} \geqs 0$ and hence $\cK^{-\frac12} \in \cL(L^2)$ is a 
well defined operator. 
Combined with \eqref{eq:limitzero} this gives $f_n = \cK^{-\frac12} \cK^{\frac12} f_n \to 0$ in $L^2$ which contradicts 
the assumption $\| f_n \|_{L^2} = 1$ for all $n \geqs 1$. Hence $\cK$ cannot be invertible on $L^2$.
\end{proof}

The following theorem is the main result on the operator equation \eqref{eq:operatoreq3}
for non-stationary $\gsp$s with non-commuting covariance operators. 
It is formulated in terms of Weyl symbols of covariance operators, using the relation 
\eqref{eq:KernelWeylsymbol} between Schwartz kernels and Weyl symbols for operators (cf. Remark \ref{rem:Fouriertransform}). 
The result has a short proof since it is a consequence of results in \cite{Wahlberg1}, with clarified assumptions
using \eqref{eq:lowerboundL2} and Lemma \ref{lem:lowerbound}. 

For Gabor expansions of $a \in M_{1 \otimes \omega}^{\infty,1}(\rr {2d})$ we use the notation 
\begin{equation}\label{eq:gaborexp1}
a = \sum _{{\bm \Lambda} \in \Theta} g_a ( {\bm \Lambda} )\Pi( {\bm \Lambda} ) \Phi, 
\quad g_a( {\bm \Lambda} ) = ( a,\Pi({\bm \Lambda} ) \widetilde{\Phi} )
\end{equation}
where $\{ g_a ({\bm \Lambda} ) \}_{ {\bm \Lambda} \in \Theta }$ are the Gabor coefficients for $a$, 
$\Phi$ is the Gaussian \eqref{eq:gaussianwindow}, 
$\wt \Phi \in \cS(\rr {2d})$ is its canonical dual window,
and $\Theta = \{ (a n, b k) \}_{n,k \in \zz {2d}} \subseteq \rr {4d}$ is a lattice determined by $a, b > 0$
that satisfy $a b < \pi$, cf. \eqref{eq:gaborexp}. 

\begin{thm}\label{thm:filternonstat1}
Suppose $u,w  \in \cL ( M^1(\rr d), L_0^2(\Omega) )$ are uncorrelated tempered $\gsp$s
with covariance operators $\cK_u, \cK_w: M^1(\rr d) \to M^\infty(\rr d)$. 
Suppose the Weyl symbols of $\cK_u$ and $\cK_w$ satisfy $a_u, a_w \in M_{1 \otimes \omega}^{\infty,1}(\rr {2d})$
with $\omega(X) = \eabs{X}^r$ for some $r\geqs 0$, and suppose 
\begin{equation}\label{eq:lowerboundL2signoise}
\left( (\cK_u + \cK_w) f, f \right) \geqs \ep \| f \|_{L^2(\rr d)}^2, \quad f \in L^2(\rr d), 
\end{equation}
holds for some $\ep > 0$. 
Then the following holds: 

\begin{enumerate}[\rm (i)]

\item 
The operator equation \eqref{eq:operatoreq3} is solved by 
\begin{equation}\label{eq:solopereq1}
F = \cK_u (\cK_u + \cK_w)^{-1} = a^w(x,D)
\end{equation}
where $a \in M_{1 \otimes \omega}^{\infty,1}(\rr {2d})$, and 
$a^w(x,D): M_{t,s}^{p,q}(\rr d) \to M_{t,s}^{p,q}(\rr d)$
is continuous for all $1 \leqs p,q \leqs \infty$ and all $t,s \in \ro$ such that $\max(|t|,|s|) \leqs \frac12 r$. 

\item 
The Gabor coefficients 
$\{ g_a ( {\bm \Lambda} ) \}_{ {\bm \Lambda} \in \Theta}$ 
of $a$ can be expressed in those of $a_u$ as  
\begin{equation}\label{eq:matrixmult}
g_a = M(g_b) \cdot g_{a_u}
\end{equation}
where $M(g_b)$ is an infinite matrix defined on $\Theta \times \Theta$, 
depending linearly on the Gabor coefficients $\{ g_b ( {\bm \Lambda} ) \}_{ {\bm \Lambda} \in \Theta}$ of 
$b \in M_{1 \otimes \omega}^{\infty,1}(\rr {2d})$
where $b^w(x,D) = (\cK_u + \cK_w)^{-1}$. 

\item 
If $r > 2d$ then the matrix $M(g_b)$ has off-diagonal decay 
\begin{equation}\label{eq:matrixdecay}
| M(g_b)({\bm \Lambda},{\bm \Omega}) | 
\lesssim \langle {\bm
\Lambda}-{\bm \Omega} \rangle^{-t}, \quad {\bm \Lambda},{\bm \Omega}
\in \Theta, 
\end{equation}
for any $t < r/2 - d$. 

\end{enumerate}

\end{thm}

\begin{proof}
The assumptions, \eqref{eq:contpsdomodsp2} and the surrounding discussion imply that 
\begin{equation*}
\cK_u, \cK_w: M_{t,s}^{p,q}(\rr d) \to M_{t,s}^{p,q}(\rr d)
\end{equation*}
are continuous for all $p,q \in [1,\infty]$, and all $t,s \in \ro$ such that $\max(|t|,|s|) \leqs \frac12 r$. 
In particular $\cK_u$ and $\cK_w$ are continuous on $L^2(\rr d) = M^2(\rr d)$, 
and hence $\cK_u + \cK_w$ is continuous on $L^2$. 
By the assumption \eqref{eq:lowerboundL2signoise} and Lemma \ref{lem:lowerbound} the operator $\cK_u + \cK_w$ is invertible on $L^2$. 

It follows from Gr\"ochenig's spectral invariance theorem 
\cite[Theorem~4.6]{Grochenig3} that $(\cK_u + \cK_w)^{-1} = b^w(x,D)$
where $b \in M_{1 \otimes \omega}^{\infty,1}(\rr {2d})$. 
From the algebraic result
\cite[Theorem~4.3]{Grochenig3} we get the conclusion that $F = \cK_u (\cK_u + \cK_w)^{-1} 
= a_u^w(x,D) b^w(x,D) =  a^w(x,D)$
with $a \in M_{1 \otimes \omega}^{\infty,1}(\rr {2d})$. 
It again follows from \eqref{eq:contpsdomodsp2} and the surrounding discussion that 
$a^w(x,D): M_{t,s}^{p,q}(\rr d) \to M_{t,s}^{p,q}(\rr d)$
is continuous for all 
$1 \leqs p,q \leqs \infty$ and all $t,s \in \ro$ such that $\max(|t|,|s|) \leqs \frac12 r$.
We have shown \eqref{eq:solopereq1} and statement (i). 

Statement (ii) follows from \cite[Section~4]{Wahlberg1} (cf. also \cite{Chen1}). 
Indeed let 
\begin{equation*}
\{ g_a ({\bm \Lambda}) \}_{{\bm \Lambda} \in \Theta}, \quad
 \{ g_b ({\bm \Lambda}) \}_{{\bm \Lambda} \in \Theta}, \quad 
 \{ g_{a_u} ({\bm \Lambda}) \}_{{\bm \Lambda} \in \Theta}
\end{equation*}
denote the Gabor coefficients defined as in \eqref{eq:gaborexp1} for for $a, b, a_u$, respectively.  
By \cite[Proposition~7]{Wahlberg1} these Gabor coefficients  are related as in \eqref{eq:matrixmult}, 
that is 
\begin{equation*}
g_a ( {\bm \Lambda} ) = 
\sum_{{\bm \Omega} \in \Theta} M(g_b) ( {\bm \Lambda}, {\bm \Omega} ) \, g_{a_u} ( {\bm \Omega} ), \quad {\bm \Lambda} \in \Theta,
\end{equation*}
where 
\begin{equation*}
M (g_b) ({\bm \Lambda},{\bm \Omega}) 
= \sum_{{\bm \Gamma} \in \Theta}
{\mathcal M}({\bm \Omega},{\bm \Gamma},{\bm \Lambda}) g_b ( {\bm
\Gamma}), \quad {\bm \Lambda},{\bm \Omega} \in \Theta, 
\end{equation*}
is a $\co$-valued infinite matrix indexed by $\Theta \times \Theta \subseteq \rr {8d}$ depending 
linearly on the Gabor coefficients $\{ g_b ({\bm \Gamma})\}_{ {\bm \Gamma} \in \Theta}$ of $b$, and 
defined in terms of 
\begin{align*}
& {\mathcal M}({\bm \Omega},{\bm \Gamma},{\bm \Lambda}) \\
& = (2 \pi)^{d} \pi^{\frac{d}{2}} \exp \Big( i \left(
\sigma(\Omega+\Omega'+\Gamma-\Gamma',\Lambda') +
\sigma(\Omega'+\Gamma',\Omega+\Gamma) \right) \Big) \\
& \qquad \times \exp \left( -\frac1{4}\left|\Omega-\Omega'-\Gamma-\Gamma'
\right|^2 \right) \mathcal V_{\wt{\Phi}} \Phi \left( \Lambda-
\frac{\Omega+\Omega'+\Gamma-\Gamma'}{2},
\Lambda'-\frac{\Omega+\Omega'- \Gamma+\Gamma'}{2} \right), \\
& \quad{\bm \Omega}, {\bm \Gamma},  {\bm \Lambda} \in \Theta,
\end{align*}
with notation as in \eqref{eq:lattice}, and where
\begin{equation*}
\mathcal V_{\wt{\Phi}} \Phi(X,Y) = (2 \pi)^{-d} ( \Phi, \Pi(X,Y) \wt \Phi ), \quad X, Y \in \rr {2d}, 
\end{equation*}
denotes a symplectic version of the STFT \eqref{eq:STFT}, cf. \eqref{eq:symplecticTFshift}, 
with $\wt \Phi$ the canonical dual window of $\Phi$.

It remains to show statement (iii). 
The assumption $r > 2 d$ admits the use of \cite[Corollary~11]{Wahlberg1}
which says that the matrix $M (g_b)$ has off-diagonal decay as in 
\eqref{eq:matrixdecay} for any $t < r/2 - d$. 
\end{proof}

The conclusion \eqref{eq:matrixmult} with the matrix $M (g_b)$ having off-diagonal decay as 
in \eqref{eq:matrixdecay} may be regarded as conceptually reminiscent of the conclusions in Theorem \ref{thm:optimalwss}, 
and in \eqref{eq:optspectralfunc} and \eqref{eq:optimalcommute} respectively.
In fact \eqref{eq:matrixmult} can be seen almost as a pointwise multiplication when the matrix has rapid off-diagonal decay, 
of the Gabor coefficients of the Weyl symbol of $\cK_u$
with a matrix that depends linearly on the Gabor coefficients of the Weyl symbol of $b^w(x,D) = (\cK_u + \cK_w)^{-1}$. 

An engineering point of view for optimal filters for non-stationary signals is treated in
\cite{Hlawatsch1}. 

\begin{rem}\label{rem:consequenceWSS}
The assumptions of Theorem \ref{thm:filternonstat1} and Parseval's theorem imply that 
\begin{equation}\label{eq:L2prop}
\left( (\cK_u + \cK_w) f, f \right) \asymp \| f \|_{L^2}^2 = \| \wh f \|_{L^2}^2, \quad f \in L^2(\rr d). 
\end{equation}
If $u$ and $w$ are WSS then by \eqref{eq:gspwss1} we have 
\begin{equation}\label{eq:L2spectral}
\left( (\cK_u + \cK_w) f, f \right) = \| \wh f \|_{L^2(\mu_u + \mu_w)}^2, \quad f \in L^2(\rr d). 
\end{equation}
It follows from \eqref{eq:L2prop} and \eqref{eq:L2spectral} that $L^2(\mu_u + \mu_w) = L^2(\rr d)$, and 
by the Radon--Nikodym theorem we have $\mu_u + \mu_w = f \dd x$
where $f \in L^\infty(\rr d)$ and $f^{-1} \in L^\infty(\rr d)$. 
The measure $\mu_u + \mu_w$ satisfies 
\begin{equation*}
C^{-1}\dd x (A) \leqs (\mu_u + \mu_w)(A)
\leqs C \dd x (A), \quad A \in \cB(\rr d), 
\end{equation*}
where $C > 0$ and $\dd x$ denotes Lebesgue measure. 
The measure $\mu_u + \mu_w$ is thus proportional to Lebesgue measure. 
We may conclude that $u + w$ resembles white noise in that all frequencies are supported in $\mu_u + \mu_w$, 
albeit with intensities that are uniformly upper and lower bounded rather than constant. 
\end{rem}

\begin{rem}\label{rem:Sjostrandpointwise}
Finally we extract an observation concerning the Gr\"ochenig--Sj\"ostrand space $M_{1 \otimes \eabs{\cdot}^r}^{\infty,1}(\rr d)$ for $r \geqs 0$. 
If $f \in M_{1 \otimes \eabs{\cdot}^r}^{\infty,1}(\rr d)$ is real-valued and $f(x) \geqs \ep > 0$ for all $x \in \rr d$, then $1/f \in M_{1 \otimes \eabs{\cdot}^r}^{\infty,1}(\rr d)$. 
Indeed by Lemma \ref{lem:WSSsymbolmodsp} we have $a = 1 \otimes f \in M_{1 \otimes \eabs{\cdot}^r}^{\infty,1}(\rr {2d})$,
and 
\begin{equation*}
( a^w(x,D) g, g )_{L^2}
= ( f \wh g, \wh g )_{L^2} 
\geqs \ep \| g \|_{L^2}^2, \quad g \in L^2(\rr d), 
\end{equation*}
so it follows from Lemma \ref{lem:lowerbound} that $a^w(x,D)$ is invertible on $L^2$. 
The inverse has Weyl symbol $a^w(x,D)^{-1} = b^w(x,D)$ with $b = 1 \otimes 1/f$. 
By Gr\"ochenig's spectral invariance theorem \cite[Theorem~4.6]{Grochenig3} we obtain $b \in M_{1 \otimes \eabs{\cdot}^r}^{\infty,1}(\rr {2d})$, 
and finally Lemma \ref{lem:WSSsymbolmodsp} yields the claim $1/f \in M_{1 \otimes \eabs{\cdot}^r}^{\infty,1}(\rr d)$. 
\end{rem}

\subsection{Further remarks on the operator equation}\label{subsec:opereq}

In this final subsection we make a few operator theoretic remarks on the operator equation 
\begin{equation}\label{eq:operatoreq4}
\cK_u = F ( \cK_u + \cK_w). 
\end{equation}
We assume that $\cK_u, \cK_w \in \cL( \cH )$ for a Hilbert space $\cH$,
and $\cK_u \geqs 0$ and $\cK_w \geqs 0$ on $\cH$ (cf. Section \ref{subsec:commuting}), 
but we do not assume that $\cK_u + \cK_w$ is invertible on $\cH$. 
We seek a solution $F \in \cL( \cH )$ to \eqref{eq:operatoreq4}, or equivalently, since 
$\cK_u^* = \cK_u$ and $\cK_w^* = \cK_w$, 
to the equation
\begin{equation}\label{eq:operatoreq5}
\cK_u = ( \cK_u + \cK_w) F^*.
\end{equation}

We have
\begin{equation}\label{eq:kernelinclusion}
\ker( \cK_u + \cK_w)  \subseteq \ker \cK_u \cap \ker \cK_w. 
\end{equation}
In fact if $( \cK_u + \cK_w )f = 0$ for $f \in \cH$ then
\begin{equation*}
0 = ( ( \cK_u + \cK_w) f,f) = ( \cK_u f,f) + ( \cK_w f, f)
\end{equation*}
which implies 
\begin{equation*}
0 = ( \cK_u f,f) = \| \cK_u^{\frac12} f \|_{\cH}^2 \quad \mbox{and} \quad  0 = ( \cK_w f,f) = \| \cK_w^{\frac12} f \|_{\cH}^2. 
\end{equation*}
Thus $\cK_u f = \cK_u^{\frac12} \cK_u^{\frac12} f = 0$ so $f \in \ker \cK_u$, and likewise $f \in \ker \cK_w$, 
which proves \eqref{eq:kernelinclusion}. 

From \eqref{eq:kernelinclusion} and $(\ker \cK_u)^\perp = \overline{\ran \cK_u}$ (where $\overline{X}$ denotes the closure of a subspace $X \subseteq \cH$) we obtain
\begin{equation}\label{eq:rangeinclusion1}
\ran \cK_u \subseteq \overline{\ran \cK_u } \subseteq \overline{ \ran ( \cK_u + \cK_w )}. 
\end{equation}

If we assume the strengthened inclusion 
\begin{equation}\label{eq:rangeinclusion2}
\ran \cK_u \subseteq  \ran ( \cK_u + \cK_w )
\end{equation}
then according to Douglas' lemma \cite{Douglas1,Fillmore1,Forough1}, 
the inclusion \eqref{eq:rangeinclusion2} is equivalent to the existence of $F \in \cL( \cH )$
that solves \eqref{eq:operatoreq5}. 
The filter linear operator $F$ is the unique solution to \eqref{eq:operatoreq5}
that satisfies 
\begin{equation}\label{eq:uniquenessconditions}
\begin{aligned}
\ran F^* & \subseteq \overline{ \ran( \cK_u + \cK_w ) },  \\
\ker F^* & = \ker \cK_u,  \\
\| F \| & = \sup_{\| f \| = 1} \frac{\| \cK_u f \|}{\| ( \cK_u + \cK_w) f \|}. 
\end{aligned}
\end{equation}

We may summarize this as follows. 

\begin{prop}\label{prop:boundedfilter}
Let $\cH$ be a Hilbert space, 
suppose $\cK_u, \cK_w \in \cL( \cH)$, $\cK_u, \cK_w \geqs 0$ on $\cH$,
and \eqref{eq:rangeinclusion2}.
Then there exists $F \in \cL( \cH )$ that solves \eqref{eq:operatoreq4}. 
The operator $F$ is the unique solution that satisfies \eqref{eq:uniquenessconditions}. 
\end{prop}

Note that we do not have to assume that $\cK_u + \cK_w$ is invertible, as we did in Theorem \ref{thm:filternonstat1}. 

Under the assumptions of Proposition \ref{prop:boundedfilter}
it follows from \cite[Corollary~1]{Fillmore1} and \eqref{eq:kernelinclusion} that $F$ is invertible as an operator in $\cL( \cH)$ if and only if 
\eqref{eq:rangeinclusion2} is strengthened into
\begin{equation*}
\ran \cK_u =  \ran ( \cK_u + \cK_w )
\end{equation*}
and $\ker \cK_u \subseteq \ker \cK_w$.

\begin{rem}\label{rem:closedrange}
If we assume that $\ran (\cK_u + \cK_w)$ is closed then \eqref{eq:rangeinclusion2} is a consequence of 
\eqref{eq:rangeinclusion1}. 
Then the Moore--Penrose pseudo-inverse $( \cK_u + \cK_w)^+ \in \cL( \cH )$ for $\cK_u + \cK_w$ is well defined \cite{Desoer1}. 
A solution to \eqref{eq:operatoreq4} is then $F = \cK_u ( \cK_u + \cK_w )^+$. 
\end{rem}

\section*{Acknowledgements}

The author is grateful to Fabio Nicola for helpful comments.
The author is a member of Gruppo Nazionale per l’Analisi Matematica, la Probabilit\`a e le loro Applicazioni (GNAMPA) -- Istituto Nazionale di Alta Matematica
(INdAM).

\end{document}